\documentclass[11pt]{amsart}

\usepackage{xcolor}
\usepackage[british]{babel}
\definecolor{mred}{rgb}{0.7,0,0}
\definecolor{mblue}{RGB}{034,113,179}
\definecolor{grey}{RGB}{120,120,120}
\usepackage[colorlinks=true,citecolor=mblue,linkcolor=mred,urlcolor=grey,backref=page]{hyperref}

\usepackage[letterpaper,total={6in, 8.5in},top=1.3in,asymmetric]{geometry}
\usepackage{amssymb}

\numberwithin{equation}{section}

\newtheorem{Theorem}{Theorem}[section]
\newtheorem{Lemma}[Theorem]{Lemma}
\newtheorem{Corollary}[Theorem]{Corollary}
\newtheorem{Proposition}[Theorem]{Proposition}

\theoremstyle{definition}

\theoremstyle{remark}
\newtheorem{Remark}[Theorem]{Remark}

\newcommand{\inc}{\hspace{3pt}\rule[0.5pt]{2mm}{0.5pt}\rule[0.5pt]{0.5pt}{4.5pt}\hspace{3pt}}

\def\Div{\mbox{\rm div}}
\def\Z{{\mathbb Z}}
\def\R{\mathbb{R}}

\def \e{\varepsilon}

\def \re{{\mathbb R}}

\def \C{{\mathbb C}}

\def \ov{\overline}

\def \0{\lambda_{0}}

\def\d{d}
\newcommand{\ee}{\mathrm{e}}
\renewcommand{\i}{i}
\newcommand{\higgs}{B}
\renewcommand{\Re}{\operatorname{Re}}
\newcommand{\hathiggs}{\hat{B}}
\newcommand{\D}{\mathrm{D}}
\newcommand{\area}{d\sigma}
\newcommand{\hatarea}{d\hat{\sigma}}
\newcommand{\weyl}{\D}
\newcommand{\hatweyl}{\hat{\D}}
\newcommand{\can}{K}
\newcommand{\acan}{K^*}

\begin{document}
\title[Convex Projective Surfaces with Compatible Weyl Connection]{Convex Projective Surfaces with Compatible Weyl Connection Are Hyperbolic}
\author[T.~Mettler]{Thomas Mettler}
\address{Institut f\"ur Mathematik, Goethe-Universit\"at Frankfurt, 60325 Frankfurt am Main, Germany}
\email{mettler@math.uni-frankfurt.de,mettler@math.ch}
\author[G.P.~Paternain]{Gabriel P.~Paternain}
\address{Department of Pure Mathematics and Mathematical Statistics,
University of Cambridge,
Cambridge CB3 0WB, England}
\email{g.p.paternain@dpmms.cam.ac.uk}

\date{27th February 2019}

\begin{abstract}
We show that a properly convex projective structure $\mathfrak{p}$ on a closed oriented surface of negative Euler characteristic arises from a Weyl connection if and only if $\mathfrak{p}$ is hyperbolic. We phrase the problem as a non-linear PDE for a Beltrami differential by using that $\mathfrak{p}$ admits a compatible Weyl connection if and only if a certain holomorphic curve exists. Turning this non-linear PDE into a transport equation, we obtain our result by applying methods from geometric inverse problems. In particular, we use an extension of a remarkable $L^2$-energy identity known as Pestov's identity to prove a vanishing theorem for the relevant transport equation.  
\end{abstract}

\maketitle

\section{Introduction}

A~\textit{projective structure} on a smooth manifold $M$ is an equivalence class $\mathfrak{p}$ of torsion-free connections on its tangent bundle $TM$, where two such connections are declared to be projectively equivalent if they share the same unparametrised geodesics. The set of torsion-free connections on $TM$ is an affine space modelled on the sections of $S^2(T^*M)\otimes TM$.  By a classical result of Cartan, Eisenhart, Weyl (see~\cite{spivakvol2} for a modern reference), two connections are projectively equivalent if and only if their difference is pure trace. In particular, it follows from the representation theory of $\mathrm{GL}(2,\R)$ that a projective structure on a surface $M$ is a section of a natural affine bundle of rank $4$ whose associated vector bundle is canonically isomorphic to $V=S^3(T^*M)\otimes \Lambda^2(TM)$. Choosing an orientation and Riemannian metric $g$ on $M$, the bundle $V$ decomposes into irreducible $\mathrm{SO}(2)$-bundles $V\simeq T^*M\oplus S^3_0(T^*M)$, where the latter summand denotes the totally symmetric $(0,\! 3)$ tensors on $M$ that are trace-free with respect to $g$, or equivalently, the cubic differentials with respect to the complex structure $J$ induced by $g$ and the orientation. In other words, fixing an orientation and Riemannian metric $g$ on $M$, a projective structure $\mathfrak{p}$ may be encoded in terms of a unique triple $(g,A,\theta)$, where $A$ is a cubic differential -- and $\theta$ a $1$-form on $M$. A conformal change of the metric $g \mapsto \ee^{2u}g$ corresponds to a change 
$$(g,A,\theta)\mapsto (\ee^{2u}g,\ee^{2u}A,\theta+\d u).$$ Consequently, the section $\Phi=A/d\sigma$ of $\can^2\otimes\ov{\acan}$ does only depend on the complex structure $J$. Here $d\sigma$ denotes the area form of $g$ and $\can$ the canonical bundle of $M$. In addition, we obtain a connection $\D$ on the anti-canonical bundle $\acan$ inducing the complex structure by taking the Chern connection with respect to $g$ and by subtracting twice the $(1,\! 0)$-part of $\theta$. Again, the connection $\D$ does only depend on $J$. Fixing a complex structure $J$ on $M$ thus encodes a given projective structure $\mathfrak{p}$ in terms of a unique pair $(\D,\Phi)$.

There are two special cases of particular interest. Firstly, we can find a complex structure $J$ so that $\D$ is the Chern connection of a metric in the conformal class determined by $J$. This amounts to finding a complex structure for which $\theta$ is exact. Secondly, we can find a complex structure $J$ so that $\Phi$ vanishes identically. This turns out to be equivalent to $\mathfrak{p}$ containing a~\textit{Weyl connection} for the conformal structure $[g]$ determined by $J$, that is, a torsion-free connection on $TM$ whose parallel transport maps are angle preserving with respect to $[g]$. 

In~\cite{MR3144212}, it is shown that a two-dimensional projective structure $\mathfrak{p}$ does locally always contain a Weyl connection and moreover, finding the Weyl connection turns out to be equivalent to finding a holomorphic curve into a certain complex surface $Z$ fibering over $M$. Here we use this observation to rephrase the problem in terms of a non-linear PDE for a Beltrami differential. More precisely, we think of $\mathfrak{p}$ as being given on a Riemann surface $(M,J)$ in terms of $(\D,\Phi)$. We show (cf.~Proposition~\ref{ppn:mainpdecplx}) that $\mathfrak{p}$ contains a Weyl connection with respect to the complex structure defined by the Beltrami differential $\mu$ on $(M,J)$ if and only if
\begin{equation}
\D^{\prime\prime}\mu-\mu\,\D^{\prime}\mu=\Phi\mu^3+\ov{\Phi},
\label{eq:nonlinear}
\end{equation}
where $\D^{\prime}$ and $\D^{\prime\prime}$ denote the $(1,\! 0)$ -- and $(0,\! 1)$-part of $\D$. 
Since every two-dimensional projective structure locally contains a Weyl connection, the above PDE for the Beltrami differential $\mu$ can locally always be solved. Moreover, on the $2$-sphere every solution $\mu$ lies in a complex $5$-manifold of solutions, whereas on a closed surface of negative Euler characteristic the solution is unique, provided it exists, see~\cite{MR3384876} (and Corollary~\ref{cor:georig} below). 

Here we address the problem of finding a projective structure $\mathfrak{p}$ for which the above PDE has no global solution. Naturally, one might start by looking at projective structures $\mathfrak{p}$ at ``the other end'', that is, those that arise from pairs $(\D,\Phi)$ where $\D$ is the Chern connection of a conformal metric, or equivalently, those for which there exists a metric $g$ so that $\mathfrak{p}$ is encoded in terms of the triple $(g,A,0)$. This class of projective structures includes the so-called~\textit{properly convex projective structures}. A projective surface $(M,\mathfrak{p})$ is called properly convex if it arises as a quotient of a properly convex open set $\Omega\subset\mathbb{RP}^2$ by a free and cocompact action of a group $\Gamma\subset \mathrm{SL}(3,\R)$ of projective transformations. In particular, using the Beltrami--Klein model of two-dimensional hyperbolic geometry, it follows that every closed hyperbolic Riemann surface is a properly convex projective surface. Motivated by Hitchin's generalisation of Teichm\"uller space~\cite{MR1174252}, Labourie~\cite{MR2402597} and Loftin~\cite{MR1828223} have shown independently that on a closed oriented surface $M$ of negative Euler characteristic every properly convex projective structure arises from a unique pair $(g,A,0)$, where $g$ and $A$ are subject to the equations
$$
K_g=-1+2|A|^2_g \quad \text{and}\quad \ov{\partial} A=0. 
$$
Using quasilinear elliptic PDE techniques, C.P.~Wang previously showed~\cite{MR1178538} (see also~\cite{MR3432157}) that the metric $g$ is uniquely determined in terms of $([g],A)$ by the equation for the Gauss curvature $K_g$ of $g$. Consequently, Labourie, Loftin conclude that on $M$ the properly convex projective structures are in bijective correspondence with pairs $([g],A)$ consisting of a conformal structure and a cubic holomorphic differential. 

Naturally one might speculate that~\eqref{eq:nonlinear} does not admit a global solution for a properly convex projective structure $\mathfrak{p}$ unless $A$ vanishes identically, in which case $\mathfrak{p}$ is hyperbolic. This is indeed the case:
\setcounter{section}{6}
\setcounter{Theorem}{1}
\begin{Corollary}
Let $(M,\mathfrak{p})$ be a closed oriented properly convex projective surface with $\chi(M)<0$ and with $\mathfrak{p}$ containing a Weyl connection $\weyl$. Then $\mathfrak{p}$ is hyperbolic and moreover $\weyl$ is the Levi-Civita connection of the hyperbolic metric.  
\end{Corollary}

This corollary is an application of the more general vanishing Theorem~\ref{thm:noweylV3} (see below) whose proof makes use of a remarkable $L^2$-energy identity. This energy identity -- known for geodesic flows as {\it Pestov's identity} -- is ubiquitous when solving uniqueness problems for X-ray transforms, including tensor tomography. To make the bridge between \eqref{eq:nonlinear} and this circle of ideas, it is necessary to recast the non-linear PDE in dynamical terms as a transport problem. Given a projective structure $\mathfrak{p}$ captured by the triple $(g,A,\theta)$ we associate a dynamical system on the unit tangent bundle $\pi:SM\to M$ of $g$ as follows. We consider a vector field of the form $F=X+(a-V\theta) V$, where $X,V$ denote the geodesic -- and vertical vector field of $SM$, $a\in C^{\infty}(SM,\mathbb{R})$ represents the cubic differential $A$ (essentially its imaginary part) and where we think of $\theta$ as a function on $SM$. The flow of the vector field $F$ is a {\it thermostat} (cf. Section \ref{section:thermo} below for more details) and it has the property that its orbits project to $M$ as unparametrised geodesics of $\mathfrak{p}$.
We show that \eqref{eq:nonlinear} is equivalent to the transport equation (cf. Corollary \ref{cor:transportweyl})
\begin{equation}
Fu=Va+\beta
\label{eq:transportintrod}
\end{equation}
on $SM$, where the real-valued function $u$ encodes a conformal metric of the sought after complex structure $\hat{J}$ and $\beta$ is a $1$-form on $M$, again thought of as a function on $SM$. Explicitly
$$
u=\frac{3}{2}\log\left(\frac{p}{(pq-r^2)^{2/3}}\right),
$$
where $p,q,r$ are given in terms of a $\hat{J}$-conformal metric $\hat{g}$ and the complex structure $J$ of $(M,g)$ by
$$
p(x,v)=\hat{g}(v,v), \qquad r(x,v)=\hat{g}(v,Jv), \quad\text{and}\quad q(x,v)=\hat{g}(Jv,Jv). 
$$ 
The right hand side in \eqref{eq:transportintrod} has degree 3 in the velocities and the dynamics of $F$ is Anosov when $\mathfrak{p}$ is a properly convex projective structure \cite{arXiv:1706.03554}, hence it is natural to think that techniques from tensor tomography might work.  Regular tensor tomography involves the geodesic vector field $X$ and the typical question at the level of the transport equation is the following: if $Xu=f$ where $f$ has degree $m$ in the velocities, is it true that $u$ has degree $m-1$ in the velocities? The case $m=2$ is perhaps the most important and it is at the core of spectral rigidity of negatively curved manifolds and Anosov surfaces~\cite{MR1632920,MR579579,MR3263517}.  Thermostats introduce new challenges, however we are able to successfully use a general $L^2$ energy identity developed in \cite{MR2486586} (following earlier results for geodesic flows in \cite{MR1863022}) together with ideas in \cite{arXiv:1706.03554}
to show that if equation \eqref{eq:transportintrod} holds then $a=0$ and $\beta$ is exact. Our vanishing Theorem~\ref{thm:noweylV3} is actually rather general and it applies to a class of projective structures considerably larger than properly convex projective structures, see Corollary~\ref{cor:nonweylminlag} below.

For the case of surfaces with boundary a full solution to the tensor tomography problem was given in \cite{MR3069117}; the solution was inspired by the proof of the Kodaira vanishing theorem in Complex Geometry. In the present paper, we go in the opposite direction, we import ideas from geometric inverse problems, to solve an existence question for a non-linear PDE in Complex Geometry. These connections were not anticipated, and it is natural to wonder if they are manifestations of something deeper.

\subsection*{Acknowledgements} The authors are grateful to Nigel Hitchin for helpful conversations. A part of the research for this article was carried out while TM was visiting FIM at ETH Z\"urich. TM thanks FIM for its hospitality. TM was partially funded by the priority programme SPP 2026 ``Geometry at Infinity'' of DFG. GPP was partially supported by EPSRC grant EP/R001898/1.

\setcounter{section}{1}
\section{Preliminaries}

Here we collect some standard facts about Riemann surfaces and the unit tangent bundle that will be needed throughout the paper. 

\subsection{The frame bundle} Throughout the article $M$ will denote a connected oriented smooth surface with empty boundary. Unless stated otherwise, all maps are assumed to be smooth, i.e., $C^{\infty}$. Let  $\pi : P \to M$ denote the~\textit{oriented frame bundle} of $M$ whose fibre at a point $x \in M$ consists of the linear isomorphisms $f : \R^2 \to T_xM$ that are orientation preserving, where we equip $\R^2$ with its standard orientation. The Lie group $\mathrm{GL}^+(2,\R)$ acts transitively from the right on each fibre by the rule $R_h(f)=f\circ h$ and this action turns $\pi : P \to M$ into a principal right $\mathrm{GL}^+(2,\R)$-bundle. The bundle $P$ is equipped with a tautological $\R^2$-valued $1$-form $\omega=(\omega^i)$ defined by $\omega_f=f^{-1}\circ d\pi_f$ and which satisfies the equivariance property $R_h^*\omega=h^{-1}\omega$. The components of $\omega$ are a basis for the $1$-forms on $P$ that are semibasic for the projection $\pi : P \to M$, i.e., those $1$-forms that vanish when evaluated on a vector field that is tangent to the fibres of $\pi : P \to M$. Therefore, if $g$ is a Riemannian metric on $M$, there exist unique real-valued functions $g_{ij}=g_{ji}$ on $P$ so that $\pi^*g=g_{ij}\omega^i\otimes \omega^j$. The Levi-Civita connection ${}^g\nabla$ of $g$ corresponds to the unique connection form $\psi=(\psi^i_j) \in \Omega^1(P,\mathfrak{gl}(2,\R))$ satisfying the structure equations
\begin{align}\label{eq:struceqlevicivita}
\begin{split}
\d\omega^i&=-\psi^i_j\wedge\omega^j,\\
\d g_{ij}&=g_{ik}\psi^k_j+g_{kj}\psi^k_i.
\end{split}
\end{align}
The curvature $\Psi=(\Psi^i_j)$ of $\psi$ is the $2$-form
$$
\Psi^i_j=\d\psi^i_j+\psi^i_k\wedge\psi^k_j=K_gg_{jk}\omega^i\wedge\omega^k,
$$
where $K_g$ denotes (the pullback to $P$ of) the Gauss curvature of $g$.

\subsection{Conformal connections}

The~\textit{conformal frame bundle} of the conformal equivalence class $[g]$ of $g$ is the principal right $\mathrm{CO}(2)$-subbundle $\pi : P_{[g]} \to M$ defined by
$$
P_{[g]}=\left\{f \in P : g_{11}(f)=g_{22}(f)\;\land \;g_{12}(f)=0\right\}.
$$
Here $\mathrm{CO}(2)=\R^+\times \mathrm{SO}(2)$ denotes the linear conformal group whose Lie algebra $\mathfrak{co}(2)$ is spanned by the matrices
$$
\begin{pmatrix} 1 & 0 \\ 0 & 1\end{pmatrix}\quad \text{and}\quad \begin{pmatrix} 0 & -1 \\ 1 & 0\end{pmatrix}. 
$$
A~\textit{conformal connection} for $[g]$ is principal $\mathrm{CO}(2)$ connection 
$$
\kappa=\begin{pmatrix} \kappa_1 & -\kappa_2 \\ \kappa_2 & \kappa_1\end{pmatrix}, \quad \kappa_i \in \Omega^1(P_{[g]})
$$ on $P_{[g]}$ which is~\textit{torsion-free}, that is, satisfies
\begin{equation}\label{eq:deftorsionfreeconfcon}
\d\begin{pmatrix}\omega^1 \\ \omega^2\end{pmatrix}=-\begin{pmatrix} \kappa_1 & -\kappa_2 \\ \kappa_2 & \kappa_1\end{pmatrix}\wedge\begin{pmatrix}\omega^1 \\ \omega^2\end{pmatrix}.
\end{equation}
The standard identification $\R^2\simeq \C$ gives an identification $\mathrm{CO}(2)\simeq \mathrm{GL}(1,\C)$ and consequently, $\mathfrak{co}(2)\simeq \C$. In particular,~\eqref{eq:deftorsionfreeconfcon} takes the form $\d \omega=-\kappa\wedge\omega$ where we think of $\kappa$ and $\omega$ as being complex-valued. Writing $r\ee^{\i\phi}$ for the elements of $\mathrm{CO}(2)$, the equivariance property for $\omega$ implies $(R_{r\\e^{\i\phi}})^*\omega=\frac{1}{r}\ee^{-\i\phi}\omega$. In particular, we see that the $\pi$-semibasic complex-valued $1$-form $\omega$ is well-defined on $M$ up to complex scale. It follows that there exists a unique complex-structure $J$ on $M$ whose $(1, \! 0)$-forms are represented by smooth complex-valued functions $u$ on $P_{[g]}$ satisfying the equivariance property $(R_{r\\e^{\i\phi}})^*u=r\ee^{\i\phi}u$, that is, so that $u\omega$ is invariant under the $\mathrm{CO}(2)$-right action. Of course, this is the standard complex structure on $M$ obtained by rotation of a tangent vector $v$ counter-clockwise by $\pi/2$ with respect to $[g]$. Denoting the canonical bundle of $M$ with respect to $J$ by $\can$, it follows that the sections of $L_{m,\ell}:=\can^{m}\otimes \ov{\can^{\ell}}$ are in one-to-one correspondence with the smooth complex-valued functions $u$ on $P_{[g]}$ satisfying the equivariance property $(R_{r\ee^{\i\phi}})^*u=r^{m+\ell}\ee^{i(m-\ell)\phi}u$. Infinitesimally, this translates to the existence of unique smooth complex-valued functions $u^{\prime}$ and $u^{\prime\prime}$ on $P_{[g]}$ so that
\begin{equation}\label{eq:struceqsecmell}
\d u=u^{\prime}\omega+u^{\prime\prime}\ov{\omega}+mu\kappa+\ell u\ov{\kappa}. 
\end{equation}

Recall, if $\alpha$ is a $1$-form on $M$ taking values in some complex vector bundle over $M$, the decomposition $\alpha=\alpha^{\prime}+\alpha^{\prime\prime}$ of $\alpha$ into its $(1,\! 0)$ part $\alpha^{\prime}$ and $(0, \! 1)$ part $\alpha^{\prime\prime}$ is given by
$$
\alpha^{\prime}=\frac{1}{2}(\alpha-i J\alpha)\quad \text{and}\quad \alpha^{\prime\prime}=\frac{1}{2}(\alpha+i J\alpha)
$$
where we define $(J\alpha)(v):=\alpha(Jv)$ for all tangent vectors $v\in TM$. The principal $\mathrm{CO}(2)$-connection $\kappa$ induces a connection on all (real or complex) vector bundles associated to $P_{[g]}$  and -- by standard abuse of notation -- we use the same letter $\D$ to denote the induced connection on the various bundles. If $s$ is the section of $L_{m,\ell}$ represented by the function $u$ satisfying~\eqref{eq:struceqsecmell}, then $\D^{\prime}s:=(\D s)^{\prime}$ is represented by $u^{\prime}$ and $\D^{\prime\prime}s:=(\D s)^{\prime\prime}$ is represented by $u^{\prime\prime}$. 

Since $\d g_{11}=\d g_{22}$ and $\d g_{12}=0$ on $P_{[g]}$, it follows from~\eqref{eq:struceqlevicivita} that the pullback of the Levi-Civita connection $\psi$ of $g$ to $P_{[g]}$ is a conformal connection. The difference of any two principal $\mathrm{CO}(2)$-connections is $\pi$-semibasic. Therefore, any other torsion-free principal $\mathrm{CO}(2)$-connection $\kappa$ on $P_{[g]}$ is of the form $\kappa=\psi-2\theta_{1}\omega$ for a unique complex-valued function $\theta_1$ on $P_{[g]}$. Since $\kappa$ is a connection, it  satisfies the equivariance property $(R_{r\ee^{\i\phi}})^*\kappa=\frac{1}{r}\ee^{-\i\phi}\kappa r\ee^{\i\phi}=\kappa$ and so does $\psi$. Therefore, $2\theta_1\omega$ is invariant under the $\mathrm{CO}(2)$-right action as well and hence twice the pullback of a $(1,\! 0)$-form on $M$ which we denote by $\theta^{\prime}$. From~\eqref{eq:struceqsecmell} we see that we may think of $\kappa$ as being the connection form of the induced connection on the anti-canonical bundle $\acan$. In particular, $\psi$ may be thought of as being the connection form of the Chern connection induced by $g$ on $\acan$. By the definition of the Chern connection, it induces the complex structure of $\acan$. Since $\psi$ and $\kappa$ differ by a $(1,\! 0)$-form, $\kappa$ also induces the complex structure of $\acan$. Consequently, the conformal connections on $P_{[g]}$ are in one-to-one correspondence with the connections $\D$ on $\acan$ inducing the complex structure, that is, $\D^{\prime\prime}=\ov{\partial}_{\acan}$. 

\subsection{The unit tangent bundle}

For what follows it will be necessary to further reduce $P_{[g]}$. The unit tangent bundle
$$
SM=\left\{(x,v) \in TM : g(v,v)=1\right\}
$$
of $g$ may be interpreted as the principal right $\mathrm{SO}(2)$-subbundle of $P$ defined by
$$
SM=\left\{f \in P : g_{ij}(f)=\delta_{ij}\right\}
$$ 
On $SM$ the identities $\d g_{ij}\equiv 0$ imply the identities $\psi^1_1\equiv\psi^2_2\equiv 0$ and $\psi^1_2\equiv -\psi^2_1$, so that $\psi$ is purely imaginary.

Abusing notation by henceforth writing $\psi$ instead of $\psi^2_1$, the structure equations thus take the form
\begin{equation}\label{eq:struceq2driem}
\d\begin{pmatrix}\omega_1\\ \omega_2\end{pmatrix}=-\begin{pmatrix} 0 & -\psi \\ \psi & 0 \end{pmatrix}\wedge\begin{pmatrix}\omega_1\\ \omega_2\end{pmatrix} \quad \text{and} \quad \d\psi=-K_g\,\omega_1\wedge\omega_2,
\end{equation}
where we write $\omega_i=\delta_{ij}\omega^j$. Note that on $SM$ the $1$-forms $\omega_1,\omega_2$ take the explicit form
\begin{equation}\label{eq:omegaexplicit}
\omega_1(\xi)=g(v,d\pi(\xi))\quad\text{and}\quad \omega_2(\xi)=g(Jv,d\pi(\xi)), \quad \xi \in T_{(x,v)}SM.
\end{equation}
Furthermore, the $1$-form $\psi$ becomes
\begin{equation}\label{eq:defpsi}
\psi(\xi)=g\left(\gamma^{\prime\prime}(0),Jv\right)
\end{equation}
where $\xi \in T_{(x,v)}SM$ and $\gamma : (-\e,\e) \to SM$ is any curve with $\gamma(0)=(x,v)$, $\dot{\gamma}(0)=\xi$ and $\gamma^{\prime\prime}$ denotes the covariant derivative of $\gamma$ along $\pi\circ \gamma$.

The three $1$-forms $(\omega_1,\omega_2,\psi)$ trivialise the cotangent bundle of $SM$ and we denote by $(X,H,V)$ the corresponding dual vector fields. The vector field $X$ is the geodesic vector field of $g$, $V$ is the infinitesimal generator of the $\mathrm{SO}(2)$-action and $H$ is the horizontal vector field satisfying $H=[V,X]$. The structure equations~\eqref{eq:struceq2driem} imply the additional commutation relations
$$
[V,H]=-X\quad \text{and}\quad [X,H]=K_gV. 
$$

Following~\cite{MR579579}, we use the volume form $\Theta=\omega_1\wedge\omega_2\wedge\psi$ on $SM$ to define an inner product
$$
\langle u,v\rangle=\int_{SM}u \bar v\,\Theta
$$
for complex-valued functions $u,v$ on $SM$ and we denote by $L^2(SM)$ the corresponding space of square integrable complex-valued functions on $SM$. The structure equations~\eqref{eq:struceq2driem} and Cartan's formula imply that all vector fields $X,H,V$ preserve $\Theta$. In particular, $-i V$ is densely defined and self-adjoint with respect to $\langle \cdot,\cdot\rangle$. Consequently, we have an orthogonal direct sum decomposition into the kernels $\mathcal{H}_m$ of the operators $m\mathrm{Id}+i V$
\begin{equation}\label{eq:vertfourier}
L^2(SM)=\bigoplus_{m \in \Z}\mathcal{H}_{m}.
\end{equation}

\subsection{Weyl connections}

If $\theta$ is a $1$-form on $M$, we may write $\pi^*\theta=\theta\omega_1+V(\theta)\omega_2$, where on the right hand side we think of $\theta$ as being a real-valued function on $SM$. Therefore, $\pi^*\theta^{\prime}=\theta_1\omega$, where $\theta_1=\frac{1}{2}(\theta-\i V\theta)$ and likewise $\pi^*\theta^{\prime\prime}=\theta_{-1}\ov{\omega}$, where $\theta_{-1}=\frac{1}{2}(\theta+\i V\theta)$. On $SM$ the connection form $\kappa$ of a conformal connection thus becomes $\kappa=\i \psi-2\theta_1\omega$ or in matrix notation
\begin{equation}\label{eq:defconfconmatrix}
\kappa=\begin{pmatrix} 0 & -\psi \\ \psi & 0\end{pmatrix}+\begin{pmatrix} -\theta\omega_1-V(\theta)\omega_2 & -V(\theta)\omega_1+\theta\omega_2 \\ V(\theta)\omega_1-\theta\omega_2 & -\theta\omega_1-V(\theta)\omega_2\end{pmatrix}.
\end{equation}
Finally, without the identification $\R^2\simeq \C$, we may equivalently think of the connection form $\kappa$ as the connection form of a torsion-free connection on $TM$. Writing $\kappa$ as 
$$
\kappa=\begin{pmatrix} 0 & -\psi \\ \psi & 0 \end{pmatrix}+\begin{pmatrix} \theta\omega_1 & \theta\omega_2 \\ V(\theta)\omega_1 & V(\theta)\omega_2\end{pmatrix}-\begin{pmatrix} 2\theta\omega_1+V(\theta)\omega_2 & V(\theta)\omega_1 \\ \theta\omega_2 & \theta\omega_1+2V(\theta)\omega_2 \end{pmatrix},
$$
the reader may easily check that $\kappa$ is the connection form of 
\begin{equation}\label{eq:exprconfcon}
\weyl={}^g\nabla+g\otimes \theta^{\sharp}-\mathrm{Sym}(\theta), 
\end{equation}
where the section $\mathrm{Sym}(\theta)$ of $S^2(T^*M)\otimes TM$ is defined by the rule 
$$
\mathrm{Sym}(\theta)(v_1,v_2)=\theta(v_1)v_2+\theta(v_2)v_1
$$ 
for all tangent vectors $v_1,v_2 \in TM$. Connections of the form~\eqref{eq:exprconfcon} for $g \in [g]$ and $\theta \in \Omega^1(M)$ are known as~\textit{Weyl connections} for the conformal structure $[g]$. By construction, they preserve $[g]$, that is, the parallel transport maps are angle preserving with respect to $[g]$. Conversely, every torsion-free connection on $TM$ preserving $[g]$ is of the form~\eqref{eq:exprconfcon} for some $g \in [g]$ and $1$-form $\theta$. Summarising, we have the following folklore result:
\begin{Proposition}
On a Riemann surface $M$ with conformal structure $[g]$ the following sets are in one-to-one correspondence:
\begin{itemize}
\item[(i)] the conformal connections on $P_{[g]}$;
\item[(ii)] the connections on $\acan$ inducing the complex structure;
\item[(iii)] the Weyl connections for $[g]$. 
\end{itemize}
\end{Proposition}

\section{Projective thermostats}
\label{section:thermo}

In this section we show how to associate the triple $(g,A,\theta)$ to a given projective structure $\mathfrak{p}$. As mentioned in the introduction, the existence of such a triple is a consequence of some elementary facts about $\mathrm{SO}(2)$-representation theory and a description of projective structures as sections of a certain affine bundle over $M$ (see~\cite{arXiv:1510.01043} for a construction of $(g,A,\theta)$ in that spirit), here instead we obtain the triple as a by-product of a characterisation of~\textit{projective thermostats}.  

A (generalised)~\textit{thermostat} is a flow $\phi$ on $SM$ which is generated by a vector field of the form $F=X+\lambda V$, where $\lambda$ is a smooth real-valued function on $SM$. In this article we are mainly interested in the case where the generalised thermostat is~\textit{projective}. By this we mean that there exists a torsion-free connection $\nabla$ on $TM$ having the property that for every $\phi$-orbit $\gamma : I \to SM$, there exists a reparametrisation $\varphi : I^{\prime} \to I$ so that $\pi\circ\gamma\circ\varphi : I^{\prime} \to M$ is a geodesic of $\nabla$. 

Phrased more loosely, the orbit projections to $M$ agree with the geodesics of a projective structure $\mathfrak{p}$ on $M$. By a classical result of Cartan, Eisenhart, Weyl (see for instance~\cite[Chap.~6, Addendum 1, Prop.~17]{spivakvol2} for a modern reference), two torsion-free connections $\nabla$ and $\nabla^{\prime}$ on $TM$ are projectively equivalent if and only if there exists a $1$-form $\alpha$ on $M$ so that
$$
\nabla^{\prime}-\nabla=\mathrm{Sym}(\alpha).
$$

\subsection{A characterisation of projective thermostats}

It turns out that projective thermostats admit a simple characterisation in terms of the vertical Fourier decomposition~\eqref{eq:vertfourier} of $\lambda$. Towards this end we first show:

\begin{Lemma}\label{ppn:charprojfol}
Let $\nabla$ be a torsion-free connection on the tangent bundle $TM$ and $\varphi=(\varphi^i_j) \in \Omega^1(SM,\mathfrak{gl}(2,\R))$ its connection form. Then, up to reparametrisation, the leaves of the foliation $\mathcal{F}$ defined by $\varphi^2_1=\omega_2=0$ project to $M$ to become the geodesics of $\nabla$. Conversely, every geodesic of $\nabla$, parametrised with respect to $g$-arc length, lifts to become a leaf of $\mathcal{F}$.  
\end{Lemma}
\begin{proof}
Recall that the set of torsion-free connections on $TM$ is an affine space modelled on the  sections of $S^2(T^*M)\otimes TM$. It follows that there exists a $1$-form $\tilde{B}$ on $M$ with values in the endomorphisms of $TM$ so that $\nabla={}^g\nabla+\tilde{B}$. As we have seen, the connection form of the Levi-Civita connection of $g$ on $TM$ is 
$$
\kappa=\begin{pmatrix} 0 & -\psi \\ \psi & 0 \end{pmatrix}.
$$
Hence there exist unique real-valued function $b^i_{jk}=b^i_{kj}$ on $SM$ so that
$$
\varphi=\begin{pmatrix} 0 & -\psi \\ \psi & 0 \end{pmatrix}+\begin{pmatrix} b^1_{11}\omega_1+b^1_{12}\omega_2 & b^1_{21}\omega_1+b^1_{22}\omega_2 \\ b^2_{11}\omega_1+b^2_{12}\omega_2 & b^2_{21}\omega_1+b^2_{22}\omega_2\end{pmatrix}.
$$
Explicitly, $b^i_{jk}(v)=g(\tilde{B}(e_j)e_k,e_i)$, where we write $e_1=v$ and $e_2=J v$ for $v \in SM$. 

Let $\delta : I \to SM$ be a leaf of $\mathcal{F}$, so that $\delta^*\omega_2=0$. Writing $\gamma:=\pi \circ \delta$ and evaluating $\delta^*\omega_2$ on the standard vector field $\partial_t$ of $\R$, we obtain
$$
0=\partial_t \inc \delta^*\omega_2=g\left(\d(\pi \circ \delta)(\partial_t),J\delta(t)\right)=g(\dot{\gamma}(t),J\delta(t)),
$$
so that $\delta=f\dot{\gamma}$ for some unique $f \in C^{\infty}(I)$. Hence without loosing generality, we may assume that the leaves of $\mathcal{F}$ are of the form $\dot{\gamma}$ for some smooth curve $\gamma : I \to M$  having unit length velocity vector with respect to $g$.

By construction of $\psi$, see~\eqref{eq:defpsi}, the pullback $1$-form $\dot{\gamma}^*\psi$ evaluated on $\partial_t$ gives the function $g({}^g\nabla_{\dot{\gamma}}\dot{\gamma},J\dot{\gamma})$, hence $\dot{\gamma}^*\varphi^2_1=0$ if and only if
$$
0=g\left({}^g\nabla_{\dot{\gamma}}\dot{\gamma},J\dot{\gamma}\right)+b^2_{11}(\dot{\gamma})=g\left({}^g\nabla_{\dot{\gamma}}\dot{\gamma}+\tilde{B}(\dot{\gamma})\dot{\gamma},J\dot{\gamma}\right).
$$
It follows that there exists a function $f \in C^{\infty}(I)$ so that
$$
{}^g\nabla_{\dot{\gamma}}\dot{\gamma}+\tilde{B}(\dot{\gamma})\dot{\gamma}=\nabla_{\dot{\gamma}}\dot{\gamma}=f\dot{\gamma}. 
$$
By a standard lemma in projective differential geometry~\cite[Chap.~6, Addendum 1, Prop.~17]{spivakvol2} a smooth immersed curve $\gamma : I \to M$ can be reparametrised to become a geodesic of the torsion-free connection $\nabla$ on $TM$ if and only if there exists a smooth function $f : I \to \R$ so that $\nabla_{\dot{\gamma}}\dot{\gamma}=f\dot{\gamma}$. The claim follows by applying this lemma. 
\end{proof}
\begin{Lemma}
Suppose the thermostat $F=X+\lambda V$ is projective, then 
$$
0=\frac{3}{2}\lambda+\frac{5}{3}VV\lambda+\frac{1}{6}VVVV\lambda. 
$$
\end{Lemma}
\begin{proof}
Using notation as in the proof of Lemma~\ref{ppn:charprojfol}, we must have $F\inc \varphi^2_1=0$ and $F\inc\omega_2=0$. The latter conditions is trivially satisfied, but the former gives
$$
F\inc \varphi^2_1=\left(X+\lambda V\right)\inc\left(\psi+b^2_{11}\omega_1+b^2_{12}\omega_2\right)=\lambda+b^2_{11}=0, 
$$
so that $\lambda=-b^2_{11}$. Since the functions $b^i_{jk}$ represent a section of $S^2(T^*M)\otimes TM$, they satisfy the structure equations
$$
\d b^i_{jk}=b^i_{jl}\kappa^l_k+b^i_{lk}\kappa^l_j-b^l_{jk}\kappa^i_l, \quad \text{mod}\quad \omega_i. 
$$
In particular, from this we compute
$$
Vb^2_{11}=V\inc \d b^2_{11}=V\inc \left(2b^2_{12}-b^1_{11}\right)\psi=2b^2_{12}-b^1_{11}. 
$$
Applying $V$ again we obtain
$$
VVb^2_{11}=2b^2_{22}-3b^2_{11}-4b^1_{12}
$$
and likewise
$$
VVVVb^2_{11}=40b^1_{12}+21b^2_{11}-20b^2_{22},
$$
so that the claim follows from an elementary calculation. 
\end{proof}
\begin{Lemma}
For $\lambda \in C^{\infty}(SM)$ the following statements are equivalent:
\begin{itemize}
\item[(i)] $0=\frac{3}{2}\lambda+\frac{5}{3}VV\lambda+\frac{1}{6}VVVV\lambda$;
\item[(ii)] $\lambda \in \mathcal{H}_{-1}\oplus \mathcal{H}_{1}\oplus \mathcal{H}_{-3}\oplus \mathcal{H}_{3}$. 
\end{itemize}
\end{Lemma}
\begin{proof}
Let $\lambda \in \mathcal{H}_{-3}\oplus \mathcal{H}_{-1}\oplus \mathcal{H}_{1}\oplus \mathcal{H}_{3}$ so that we may write $\lambda=\lambda_{-3}+\lambda_{-1}+\lambda_1+\lambda_3$ with $\lambda_m \in \mathcal{H}_m$. Since $\lambda$ is real-valued we have $\lambda_{-1}=\ov{\lambda_1}$ and $\lambda_{-3}=\ov{\lambda_3}$. Hence setting $\nu_1=\lambda_{-1}+\lambda_1$ and $\nu_3=\lambda_{-3}+\lambda_3$, we obtain $VV\nu_1=-\nu_1$ and $VV\nu_3=-9\nu_3$ so that
$$
\frac{3}{2}\lambda+\frac{5}{3}VV\lambda+\frac{1}{6}VVVV\lambda=\frac{3}{2}(\nu_3+\nu_1)+\frac{5}{3}(-9\nu_3-\nu_1)+\frac{1}{6}(81\nu_3+\nu_1)=0.
$$
Conversely, suppose $\lambda \in C^{\infty}(SM)$ satisfies $0=\frac{3}{2}\lambda+\frac{5}{3}VV\lambda+\frac{1}{6}VVVV\lambda$ and write $\lambda=\sum_{m}\lambda_m$ with $\lambda_m \in \mathcal{H}_m$. Hence we obtain
$$
0=\frac{3}{2}\lambda+\frac{5}{3}VV\lambda+\frac{1}{6}VVVV\lambda
=\sum_{m}\left(\frac{3}{2}-\frac{5}{3}m^2+\frac{1}{6}m^4\right)\lambda_m
$$
so that $\lambda_m=0$ unless
$$
0=\frac{3}{2}-\frac{5}{3}m^2+\frac{1}{6}m^4=\frac{1}{6}(m-3)(m-1)(m+1)(m+3).
$$ 
The claim follows.
\end{proof}
Finally, we obtain:
\begin{Proposition}
A thermostat $F=X+\lambda V$ is projective if and only if $\lambda \in \mathcal{H}_{-1}\oplus \mathcal{H}_{1}\oplus \mathcal{H}_{-3}\oplus \mathcal{H}_{3}$. 
\end{Proposition}
\begin{proof}
It remains to show that if $\lambda \in \mathcal{H}_{-1}\oplus \mathcal{H}_{1}\oplus \mathcal{H}_{-3}\oplus \mathcal{H}_{3}$, then there exists a torsion-free connection $\nabla$ on $TM$ so that $F \inc \varphi^2_1$ vanishes identically, where $\varphi=(\varphi^i_j)$ denotes the connection form of $\nabla$. We may write 
$$
\lambda=a-V\theta
$$ 
where $a \in C^{\infty}(SM)$ satisfies $9a+VVa=0$ and $\theta$ is a smooth $1$-form on $M$, thought of as a real-valued function on $SM$. Since $9a+VVa=0$, there exists a unique cubic differential $A$ on $M$ so that $\pi^*A=(Va/3+ia)\omega^3$. Hence simple computations show that 
\begin{align}\label{eq:idcubicdiff}
\begin{split}
a(v)&=\Re A(Jv,Jv,Jv)=-\Re A(Jv,v,v)\\
\frac{1}{3}Va(v)&=\Re A(v,v,v)=-\Re A(v,Jv,Jv)
\end{split}
\end{align}
for all $v \in SM$. Let $\higgs$ be the unique $1$-form on $M$ with values in the endomorphisms of $TM$ satisfying
\begin{equation}\label{eq:defrealhiggs}
g(\higgs(v_1)v_2,v_3)=\mathrm{Re}\, A(v_1,v_2,v_3)
\end{equation}
for all tangent vectors $v_1,v_2,v_3 \in TM$. On $TM$ consider the torsion-free connection $\nabla=\weyl+\higgs$, where $\weyl$ is the Weyl connection
$$
\weyl={}^g\nabla+g\otimes \theta^{\sharp}-\mathrm{Sym}(\theta). 
$$
Using~\eqref{eq:defconfconmatrix} and~\eqref{eq:idcubicdiff}, we compute that the connection form of  $\nabla$ is
\begin{multline}\label{eq:twistedconformalcon}
\varphi=\begin{pmatrix} -\theta\omega_1-V(\theta)\omega_2 & -V(\theta)\omega_1+\theta\omega_2 -\psi \\ \psi+V(\theta)\omega_1-\theta\omega_2 & -\theta\omega_1-V(\theta)\omega_2\end{pmatrix}\\
+\begin{pmatrix} V(a)/3\omega_1-a\omega_2 & -a\omega_1-V(a)/3\omega_2 \\ -a\omega_1-V(a)/3\omega_2 & -V(a)/3\omega_1+a\omega_2\end{pmatrix}.
\end{multline}
In particular, we have
$$
\varphi^2_1=\psi+\left(V(\theta)-a\right)\omega_1-\left(\theta+V(a)/3\right)\omega_2,
$$
so that $F\inc \varphi^2_1=0$.
\end{proof}

\subsection{The effect of a conformal change}

Summarising the previous subsection, we have seen that if $\nabla$ is a torsion-free connection on $TM$ and we fix a Riemannian metric $g$ on $M$, then we may write $\nabla={}^g\nabla+\tilde{B}$ for some endomorphism-valued $1$-form $\tilde{B}$ on $M$. The thermostat on $SM$ defined by $\lambda=-b^2_{11}$ has the property that its orbits project to $M$ to become the geodesics of $\nabla$ up to parametrisation. Moreover, we obtain a $1$-form $\theta \in \Omega^1(M)$ as well as a cubic differential $A \in \Gamma(\can^3)$, so that the connection $\nabla$ shares its geodesics -- up to parametrisation -- with the projections to $M$ of the orbits of the projective thermostat defined by $\lambda=a-V\theta$, where $a$ represents the imaginary part of $A$.  

Next we compute how $\theta$ and $A$ transform under conformal change of the metric. As a consequence, we obtain:
\begin{Proposition}\label{prop:uniquerep}
Let $\nabla$ be a torsion-free connection on $TM$. Then the choice of a conformal structure $[g]$ on $M$ determines a unique Weyl connection $\weyl$ for $[g]$ and a unique section $\Phi$ of $K^2\otimes\ov{\acan}$ so that $\weyl+\Re\Phi$ is projectively equivalent to $\nabla$. 
\end{Proposition}
\begin{proof}
Let $g \mapsto \hat{g}=\ee^{2u}g$ be a conformal change of the metric, where $u \in C^{\infty}(M)$. For the new metric $\hat{g}$ there exists a $1$-form $\hat{\theta}$ and a cubic differential $\hat{A}$ on $M$ so that $\weyl+\higgs$ and $\hatweyl+\hathiggs$ are projectively equivalent. Here $\hathiggs$ denotes the $1$-form constructed from $\hat{A}$ by using the metric $\hat{g}$. Projective equivalence corresponds to the existence of a $1$-form $\alpha$ on $M$ so that
$$
\weyl+\higgs=\hatweyl+\hathiggs+\mathrm{Sym}(\alpha)
$$
Using~\eqref{eq:exprconfcon} as well as (see~\cite[Theorem 1.159]{MR867684})
\begin{equation}\label{eq:levcivconfchange}
{}^{\exp(2u)g}\nabla={}^g\nabla-g\otimes{}^g\nabla u+\mathrm{Sym}(\d u)
\end{equation}
this is equivalent to
\begin{multline*}
{}^g\nabla+g\otimes\theta^{\sharp}-\mathrm{Sym}(\theta)+\higgs={}^g\nabla-g\otimes{}^g\nabla u\\+\mathrm{Sym}(\d u)+\ee^{2u}g\otimes\hat{\theta}^{\hat{\sharp}}-\mathrm{Sym}(\hat{\theta})+\hathiggs+\mathrm{Sym}(\alpha)
\end{multline*}
or
$$
g\otimes\left(\theta^{\sharp}+{}^g\nabla u-\hat{\theta}^{\sharp}\right)+\higgs-\hathiggs=\mathrm{Sym}\left(\beta\right),
$$
where $\beta=\alpha+\theta+\d u-\hat{\theta}$. Evaluating this equation on the pair $(v,Jv)$ with $v$ a unit tangent vector with respect to $g$ gives 
$$
\higgs(v)Jv-\hathiggs(v)Jv=\mathrm{Sym}\left(\beta\right)(v,Jv).
$$
Computing the inner product with the tangent vector $v$ yields
$$
\mathrm{Re}\,A(v,Jv,v)-\ee^{-2u}\mathrm{Re}\,\hat{A}(v,Jv,v)=\beta(Jv).
$$
Thought of as an identity for functions on $SM$, the left hand side lies in $\mathcal{H}_{-3}\oplus\mathcal{H}_3$ whereas the right hand side lies in $\mathcal{H}_{-1}\oplus \mathcal{H}_1$ and hence they can only be equal if both sides vanish identically. Consequently, it follows that $\beta=0$ and that 
$$\hat{A}=\ee^{2u}A.
$$ 
Therefore, $\higgs=\hathiggs$ and 
\begin{equation}\label{eq:confchange1form}
\hat{\theta}=\theta+\d u
\end{equation} so that $\alpha=0$ as well as $\weyl=\hatweyl$. 

In particular, we see that both $\weyl$ and $B$ do only depend on the conformal equivalence class of $g$. We may define a section $\Phi$ of $K^2\otimes \ov{\acan}$ by $\Phi\area=A$, where $\area$ denotes the area form of $g$. Comparing with~\eqref{eq:defrealhiggs}, we see that $B$ is the real part of $\Phi$. 
\end{proof}

\section{Holomorphic curves}

It is natural to ask whether for a given torsion-free connection $\nabla$ on $TM$ one can always (at least locally) choose a conformal structure $[g]$ on $M$ so that $\Phi$ vanishes identically. Equivalently, whether every torsion-free connection $\nabla$ on $TM$ is locally projectively equivalent to a Weyl connection $\weyl$. This question was answered in the affirmative in~\cite{MR3144212}, where it is also observed that the problem is equivalent to finding a suitable holomorphic curve into a complex surface fibering over $M$. Here we will briefly review this observation and use it do derive a non-linear PDE for the Beltrami differential of the sought after conformal structure. 

\begin{Remark}
Given that one can locally always find a conformal structure so that $\Phi$ vanishes identically, one might wonder whether it is possible to simultaneously pick a conformal metric so that the $1$-form $\theta$ is closed. Indeed,~\eqref{eq:levcivconfchange} and~\eqref{eq:confchange1form} imply that the additional closedness condition corresponds to $\nabla$ being locally  projectively equivalent to a Levi-Civita connection of some metric. However, this is not always possible, see~\cite{MR2581355}.
\end{Remark}

\subsection{A complex surface}

Inspired by the twistorial construction of holomorphic projective structures by Hitchin~\cite{MR699802}, it was shown in~\cite{MR728412} and~\cite{MR812312} and how to construct a ‘twistor space‘ for smooth projective structures. Let $\nabla$ be a torsion-free connection on $TM$ and $\varphi=(\varphi^i_j) \in \Omega^1(P,\mathfrak{gl}(2,\R))$ its connection form on the frame bundle $P$. We can use $\varphi$ to construct a complex structure on the quotient $P/\mathrm{CO}(2)$. By definition, an element of $P/\mathrm{CO}(2)$ gives a frame in some tangent space of $M$, well defined up to rotation and scaling. Therefore, the conformal structures on $M$ are in one-to-one correspondence with the sections of the fibre bundle $P/\mathrm{CO}(2) \to M$ whose fibre is $\mathrm{GL}^+(2,\R)/\mathrm{CO}(2)$, that is, the open disk. We will construct a complex structure on $P/\mathrm{CO}(2)$ in terms of its $(1,\! 0)$-forms, or more precisely, the pullbacks of the $(1,\! 0)$-forms to $P$. Recall that the Lie algebra $\mathfrak{co}(2)$ of $\mathrm{CO}(2)$ is spanned by the matrices
$$
\begin{pmatrix} 1 & 0 \\ 0 & 1\end{pmatrix}\quad \text{and}\quad \begin{pmatrix} 0 & -1 \\ 1 & 0\end{pmatrix}.
$$
Consequently, the complex-valued $1$-forms on $P$ that are semibasic for the quotient projection $P \to P/\mathrm{CO}(2)$ are spanned by the form $\omega$ and 
$$
\zeta=(\varphi^1_1-\varphi^2_2)+i (\varphi^1_2+\varphi^2_1)
$$ 
as well as their complex conjugates. Recall that we have $\left(R_{r\mathrm{e}^{i\phi}}\right)^*\omega=\frac{1}{r}\mathrm{e}^{-i\phi}$ and using that $\varphi$ satisfies the equivariance property $R_h^*\varphi=h^{-1}\varphi h$ for all $h \in \mathrm{GL}^+(2,\R)$, we compute $\left(R_{r\mathrm{e}^{i\phi}}\right)^*\zeta=\mathrm{e}^{-2i\phi}\zeta$. 
It follows that there exists a unique almost complex structure $J$ on $P/\mathrm{CO}(2)$ whose $(1,\!0)$-forms pull back to $P$ to become linear combinations of the forms $\omega,\zeta$. The almost complex structure $J$ can be shown to only depend on the projective equivalence class of $\nabla$ and moreover, an application of the Newlander--Nirenberg theorem shows that $J$ is always integrable, see~\cite{MR3144212} for details. 

\subsection{M\"obius action}

In our setting it is convenient to reduce the frame bundle $P$ to the unit tangent bundle $SM$ of some fixed metric $g$. In order to get a handle on the complex surface $P/\mathrm{CO}(2)$ after having carried out this reduction, we interpret the disk bundle $P/\mathrm{CO}(2) \to M$ as an associated bundle to the frame bundle $P$. This requires an action of the structure group $\mathrm{GL}^+(2,\R)$ on the open disk and this is what we compute next.

The group $\mathrm{GL}^+(2,\R)$ acts from the left on the lower half plane 
$$
-\mathbb{H}:=\left\{w \in \C : \Im(w)<0\right\}
$$ 
by M\"obius transformations, where $w$ denotes the standard coordinate on $\C$. We let $\mathbb{D}\subset \C$ denote the open unit disk. Identifying $-\mathbb{H}$ with $\mathbb{D}$ via the M\"obius transformation 
$$
-\mathbb{H} \to \mathbb{D}, \quad w \mapsto -\left(\frac{w+\i}{w-\i}\right)
$$
we obtain an induced action of $\mathrm{GL}^+(2,\R)$ on $\mathbb{D}$ making this transformation equivariant 
\begin{equation}\label{eq:moeb}
\begin{pmatrix} a & b \\ c & d \end{pmatrix} \cdot z=\frac{\i z(a+d)+z(b-c)-\i(a-d)+(b+c)}{-\i z(a-d)-z(b+c)+\i(a+d)-(b-c)}.
\end{equation}
The stabiliser subgroup of the point $z=0$ consists of elements in $\mathrm{GL}^+(2,\R)$ satisfying $a=d$ and $b+c=0$, i.e., the linear conformal group $\mathrm{CO}(2)$. Consequently, we have $\mathbb{D}\simeq \mathrm{GL}^+(2,\R)/\mathrm{CO}(2)$ and we obtain a projection
$$
\lambda : \mathrm{GL}^+(2,\R) \to \mathbb{D}, \quad \begin{pmatrix} a & b \\ c & d \end{pmatrix} \mapsto \begin{pmatrix} a & b \\ c & d \end{pmatrix}\cdot 0=\frac{-\i(a-d)+(b+c)}{\i(a+d)-(b-c)}.
$$
In particular, a mapping $z : N \to \mathbb{D}$ from a smooth manifold $N$ into $\mathbb{D}$ is covered by a map 
$$
\tilde{z}=\begin{pmatrix} \frac{1-|z|^2}{(1+z)(1+\ov z)} & \frac{\i(z-\ov{z})}{(1+z)(1+\ov z)} \\ 0 & 1\end{pmatrix}.
$$
into $\mathrm{GL}^+(2,\R)$. Equivalently, we have $\tilde{z}\cdot 0=z$ or $z \cdot \tilde{z}=0$, where as usual we turn the left action into a right action by the definition $z\cdot \tilde{z}:=\tilde{z}^{-1}\cdot z$. 
 
Let $\rho : Z \to M$ denote the disk-bundle associated to the above $\mathrm{GL}^+(2,\R)$ action on $\mathbb{D}$. Suppose $z : P \to \mathbb{D}$ represents a section of $Z \to M$ so that $z$ is a $\mathrm{GL}^+(2,\R)$-equivariant map.  For every coframe $u \in P$ the pair $(u,z(u)) \in P \times \mathbb{D}$ lies in the same $\mathrm{GL}^+(2,\R)$ orbit as 
\begin{equation}\label{eq:bundlemapup}
(u\cdot \tilde{z}(u),z(u)\cdot\tilde{z}(u))=(u\cdot \tilde{z}(u),0).
\end{equation}
Therefore, the map $z$ gives for every point $p \in M$ a coframe $u\cdot \tilde{z}(u)$ which is unique up to the action of $\mathrm{CO}(2)$. It follows that the bundle $Z \to M$ is isomorphic to $P/\mathrm{CO}(2) \to M$, as desired.

Let $\Upsilon : P\times \mathbb{D} \to P$ be the map defined by~\eqref{eq:bundlemapup}. We will next compute the pullback of $\omega,\zeta$ under $\Upsilon$. Note that we may write $\Upsilon=R\circ\left(\mathrm{Id}_P\times \tilde{z}\right)$ where $R : P \times \mathrm{GL}^+(2,\R) \to P$ denotes the $\mathrm{GL}^+(2,\R)$ right action of $P$. Recall the standard identities 
$$
R^*\varphi=h^{-1}\varphi h+ h^{-1}\d h\quad \text{and} \quad R^*\omega=h^{-1}\omega,
$$
where $h : P\times \mathrm{GL}^+(2,\R) \to \mathrm{GL}^+(2,\R)$ denotes the projection onto the latter factor. From this we compute
\begin{equation}\label{eq:newcplxstruc}
\omega_\Upsilon:=\Upsilon^*\omega=\tilde{z}^{-1}\omega=\left(\frac{1+\ov{z}}{1-|z|^2}\right)\left(\omega+z\ov{\omega}\right).
\end{equation}
and
\begin{equation}\label{eq:pullbackupsilon1}
\varphi_{\Upsilon}:=\Upsilon^*\varphi=\tilde{z}^{-1}\varphi \tilde{z}+\tilde{z}^{-1}\d\tilde{z}
\end{equation}
We also obtain $\zeta_{\Upsilon}=\Upsilon^*\zeta=(\varphi_{\Upsilon})^1_1-(\varphi_{\Upsilon})^2_2+\i\left((\varphi_{\Upsilon})^1_2+(\varphi_{\Upsilon})^2_1\right)$. Writing
$$
\chi=\frac{1}{2}\left(3(\varphi^1_1+\varphi^2_2)+\i(\varphi^2_1-\varphi^1_2)\right),
$$
and using~\eqref{eq:pullbackupsilon1}, a tedious but straightforward calculation gives
\begin{equation}\label{eq:pullbackzetaupsilon13}
\zeta_\Upsilon=\frac{2(1+\ov{z})}{(|z|^2-1)(z+1)}\left(\d z-\frac{1}{2}\zeta+\frac{1}{2}z^2\ov{\zeta}+z\chi-z\bar{\chi}\right). 
\end{equation}
\begin{Remark}
The complex-valued $1$-form $\chi$ is chosen so that $\chi,\ov{\chi},\omega,\ov{\omega},\zeta,\ov{\zeta}$ span the complex-valued $1$-forms on $P$. Clearly, this condition does not pin down $\chi$ uniquely. The particular choice is so that in the absence of $\theta$ the form $\chi$ becomes the connection form of the Chern connection on $\acan$ upon reducing to $SM$, see~\eqref{eq:reductionchism} below. 
\end{Remark}
The complex structure on $Z$ does only depend on the projective equvialence class of $\nabla$. Thus, after possibly replacing $\varphi$ with a projectively equivalent connection, we can assume that the torsion-free connection on $TM$ corresponding to $\varphi$ is of the form $\weyl+\higgs$ for some $1$-form $\theta$ and some cubic differential $A$ on $M$. On the unit tangent bundle $SM$ of $g$ the connection form of $\weyl+\higgs$ takes the form~\eqref{eq:twistedconformalcon}. Using this equation and reducing to $SM\subset P$ yields the following identities on $SM$
\begin{align}\label{eq:reductionchism}
\begin{split}
\zeta&=2a_{-3}\ov{\omega}, \\
\chi&=\i\psi-4\theta_1\omega-2\theta_{-1}\ov{\omega},
\end{split}
\end{align}
Recall, we write $a_3=\frac{1}{3}Va+\i a$ and $a_{-3}=\ov{a_3}$ as well as $\theta_1=\frac{1}{2}(\theta-\i V\theta)$ and $\theta_{-1}=\ov{\theta_1}$. Also, the connection form $\kappa$ of the induced Weyl connection is $\kappa=\i \psi-2\theta_1\omega$, see~\eqref{eq:defconfconmatrix}. Therefore, we have
$$
\chi=2\kappa+\ov{\kappa}.
$$

The $\mathrm{SO}(2)$-action induced by~\eqref{eq:moeb} is
$$
\begin{pmatrix}\cos \phi & -\sin\phi \\ \sin\phi & \cos\phi\end{pmatrix}\cdot z=\frac{2\i z\cos\phi-2z\sin\phi}{2\i\cos\phi+2\sin\phi}=\mathrm{e}^{2\i\phi}z
$$
and hence the equivariance property of a function $z : SM \to \mathbb{D}$ representing a section of $Z \to M$ becomes $\left(R_{\mathrm{e}^{\i\phi}}\right)^*z=\mathrm{e}^{-2\i\phi}z$, that is, $z$ represents a section of $\can^{-2}$. Since we have a metric, we have an identification $\acan\simeq \ov{\can}$ and hence $\can^{-2}\simeq \acan\otimes \ov{\can}$. In particular, we may write
\begin{equation}\label{eq:equationfordz}
\d z=z^{\prime}\omega+z^{\prime\prime}\ov\omega+\ov{\kappa}z-\kappa z
\end{equation}
For unique complex-valued functions $z^{\prime}$ and $z^{\prime\prime}$ on $SM$. 
Consequently, using \eqref{eq:pullbackzetaupsilon13}, \eqref{eq:reductionchism} and \eqref{eq:equationfordz} we obtain 
\begin{align}\label{eq:bzeta}
\begin{split}
\left(\frac{(|z|^2-1)(z+1)}{2(1+\ov{z})}\right)\zeta_\Upsilon&=z^{\prime}\omega+z^{\prime\prime}\ov{\omega}+\ov{\kappa}z-\kappa z-a_{-3}\ov{\omega}+z^2a_3\omega\\
&\phantom{=}+z(2\kappa+\ov{\kappa})-z(2\ov{\kappa}+\kappa)\\
&=\left(z^{\prime}+z^2a_3\right)\omega+\left(z^{\prime\prime}-a_{-3}\right)\ov{\omega}.
\end{split}
\end{align}
 In order to connect the expressions for $\omega_\Upsilon$ and $\zeta_\Upsilon$ to the condition of $z$ representing a conformal structure $[\hat{g}]$ that defines a holomorphic curve into $Z$, we use the following elementary lemma:
\begin{Lemma}\label{lem:holcurvlemma}
Let $Z$ be a complex surface and $\omega,\zeta \in \Omega^1(Z,\C)$ a basis for the $(1,\! 0)$-forms of $Z$. Suppose $M\subset Z$ is a smooth surface on which $\omega\wedge\ov{\omega}$ is non-vanishing. Then $M$ admits the structure of a holomorphic curve -- that is, a complex $1$-dimensional submanifold of $Z$ -- if and only if $\omega\wedge\zeta$ vanishes identically on $M$. 
\end{Lemma}
\begin{proof}
Since $\omega\wedge\ov{\omega}$ is non-vanishing on $M$, the forms $\omega$ and $\ov{\omega}$ span the complex-valued $1$-forms on $M$. Since $M$ is a complex submanifold of $Z$ if and only if the pullback of a $(1,\! 0)$-form on $Z$ is a $(1,\! 0)$-form on $M$, the claim follows. 
\end{proof}
The reduction of $P$ to $SM$ identifies $Z$ with $SM\times_{\mathrm{SO}(2)} \mathbb{D}$. Now suppose the conformal structure $[\hat{g}] : M \to Z$ is represented by the map $z : SM \to \mathbb{D}$. If $v : U \to SM$ is a local section of $\pi : SM \to M$, then $[\hat{g}]|_U : U \to Z$ is covered by the map $(\mathrm{Id}_{SM}\times z)\circ v : U \to SM \times \mathbb{D}$. Recall that the complex structure on $Z$ has the property that its $(1,\! 0)$-forms pull-back to become linear combination of $\omega_\Upsilon$ and $\zeta_\Upsilon$. Using the expressions~\eqref{eq:newcplxstruc} and~\eqref{eq:bzeta} for the pullbacks of $\omega_\Upsilon$ and $\zeta_\Upsilon$ to $SM$ we obtain
$$
\omega_\Upsilon\wedge\zeta_\Upsilon=-\frac{2(1+\ov{z})^2}{(|z|^2-1)^2(z+1)}\left(z^{\prime\prime}-zz^{\prime}-z^3a_3-a_{-3}\right)\omega\wedge\ov{\omega}.
$$
In particular, since $v : U \to SM$ is a $\pi$-section and $\omega$ and $\ov{\omega}$ are $\pi$-semibasic, the pullback $v^*(\omega_\Upsilon\wedge\zeta_\Upsilon)$ vanishes if and only if $\omega_\Upsilon\wedge\zeta_\Upsilon$ vanishes on $\pi^{-1}(U)$. Thus, Lemma~\ref{lem:holcurvlemma} implies that $z$ represents a holomorphic curve if and only if 
\begin{equation}\label{eq:beltramipdesm}
z^{\prime\prime}-zz^{\prime}=z^3a_3+a_{-3}. 
\end{equation}
\subsection{The Beltrami differential}

So far we have not explicitly tied the conformal structure $[\hat{g}]$ to the function $z : SM \to \mathbb{D}$ representing it. In order to do this we first recall the Beltrami differential. The choice of a metric $\hat{g}$ on $M$ allows to define the functions
$$
p(x,v)=\hat{g}(v,v), \qquad r(x,v)=\hat{g}(v,Jv) \quad\text{and}\quad q(x,v)=\hat{g}(Jv,Jv)
$$
on $SM$. The orientation compatible complex structure $\hat{J}$ on $M$ induced by the conformal equivalence class of $\hat{g}$ has matrix representation
$$
\hat{J}=\frac{1}{\sqrt{pq-r^2}}\begin{pmatrix} -r & -q \\ p & r\end{pmatrix}.
$$
In particular, we compute that the $(1,\! 0)$-forms with respect to $\hat{J}$ pull-back to $SM$ to become complex multiples of
\begin{equation}\label{eq:omegaj}
\omega_{\hat{J}}:=\frac{1}{2}\left(\omega-i\hat{J}\omega\right)=\left(\frac{p+q+2\sqrt{pq-r^2}}{4\sqrt{pq-r^2}}\right)\left(\omega+\mu\ov{\omega}\right)
\end{equation}
where
$$
\mu=\frac{(p-q)+2\i r}{p+q+2\sqrt{pq-r^2}}
$$
is the~\textit{Beltrami coefficient} of $\hat{J}$. Clearly, $\mu$ does only depend on the conformal equivalence class $[\hat{g}]$ of $\hat{g}$. Moreover, the function $\mu$ represents a $(0,\! 1)$-form on $M$ with values in $\acan$ called the ~\textit{Beltrami differential} of $[\hat{g}]$, which -- by abuse of language -- we denote by $\mu$ as well.

The reduction of $P$ to the unit tangent bundle $SM$ of $g$ turned $\omega$ into a basis for the $(1,\! 0)$-forms with respect to the complex structure induced by $g$ and the orientation. The mapping $z$ represents a conformal structure $[\hat{g}]$ and consequently, induces an orientation compatible complex structure $\hat{J}$ whose $(1,\! 0)$-forms we computed in~\eqref{eq:newcplxstruc}. Comparing this expression with the formula~\eqref{eq:omegaj} for the Beltrami coefficient shows that we obtain the same $(1,\! 0)$-forms if and only if $z=\mu$. Remember, $z^{\prime}$ and $z^{\prime\prime}$ represent the $(1,\! 0)$ and $(0,\! 1)$ part of the derivative of $z$ with respect to the connection $\D$ induced by the Weyl connection $\weyl$. Furthermore, the function $a_{3}$ represents the cubic differential $A$ or equivalently, the form $\Phi$, since $\Phi \area=A$ and $\area$ is represented by the constant function $1$ on $SM$. Using equation~\eqref{eq:beltramipdesm} and the fact that $\mathfrak{p}$ contains a Weyl connection with respect to $[\hat{g}]$ if and only if $[\hat{g}] : M \to Z$ is a holomorphic curve~\cite[Theorem 3]{MR3144212}, we have thus shown:
\begin{Proposition}\label{ppn:mainpdecplx}
Let $(M,[g])$ be a Riemann surface equipped with a projective structure $\mathfrak{p}$ given in terms of $(\D,\Phi)$. Then $\mathfrak{p}$ contains a Weyl connection with respect to the conformal structure defined by the Beltrami differential $\mu$ if and only if 
\begin{equation}\label{eq:pdebeltramicoeff}
\D^{\prime\prime}\mu-\mu\,\D^{\prime}\mu=\Phi\mu^3+\ov{\Phi}.
\end{equation}
\end{Proposition}
\begin{Remark}
In the special case where $\mathfrak{p}$ is a properly convex projective structure, an equation equivalent to~\eqref{eq:pdebeltramicoeff} was previously obtained by N.~Hitchin using the Higgs bundle description of $\mathfrak{p}$.\footnote{Private communication, August 2014.}
\end{Remark}
As a corollary, we obtain:
\begin{Corollary}\label{cor:georig}
Let $M$ be a closed oriented surface with $\chi(M)<0$. Suppose the Weyl connections $\weyl$ and $\hatweyl$ on $TM$ are projectively equivalent. Then $\weyl=\hatweyl$ and they preserve the same conformal structure.  
\end{Corollary}
\begin{proof}
Equip $M$ with the Riemann surface structure defined by $[g]$ and the orientation. Let $\mathfrak{p}$ be the projective structure defined by $\weyl$ (or $\hatweyl$). The projective structure $\mathfrak{p}$ is encoded in terms of the pair $(\weyl,0)$. Moreover, the Beltrami differential $\mu$ defined by $[\hat{g}]$ solves~\eqref{eq:pdebeltramicoeff}, that is,
$$
\D^{\prime\prime}\mu-\mu\D^{\prime}\mu=0. 
$$
Now observe that $\ov{\partial}_{\mu}=\D^{\prime\prime}-\mu\D^{\prime}$ defines a del-bar operator on $\ov{K}\otimes K^{*}$ and hence~~\eqref{eq:pdebeltramicoeff} can be written as $\ov{\partial}_{\mu}\mu=0$. Therefore, $\mu$ is holomorphic with respect to the holomorphic line bundle structure defined by $\ov{\partial}_{\mu}$ on $\ov{K}\otimes K^{*}$. However, since $\chi(M)<0$, the line bundle $\ov{K}\otimes K^{*}$ has negative degree, so that its only holomorphic section is the zero-section. It follows that $\mu=0$ and hence $[g]=[\hat{g}]$. Since $\weyl$ and $\hatweyl$ are projectively equivalent and preserve the same conformal structure $[g]$, we conclude exactly as in the proof of Proposition~\ref{prop:uniquerep} that $\weyl=\hatweyl$. 
\end{proof}
\begin{Remark}
The above corollary was first proved in~\cite{MR3384876}. In particular, as a special case, it also shows that on a closed surface with $\chi(M)<0$, the unparametrised geodesics of a Riemannian metric determine the metric up to rescaling by a positive constant. This was first observed in~\cite{MR1796527}.  
\end{Remark}

\section{The transport equation}

While the PDE~\eqref{eq:pdebeltramicoeff} for the Beltrami differential $\mu$ is natural from a complex geometry point of view, it turns out to be advantageous to rephrase it as a transport equation on $SM$. The relevant transport equation on $SM$ can be derived using~\eqref{eq:pdebeltramicoeff} -- see Appendix~\ref{app:1} --  but here we will instead take a different approach, as it leads to a more general result about thermostats having the same unparametrised geodesics, see Proposition~\ref{prop:transport}.

Let $g,\hat{g}$ be Riemannian metrics on $M$. In what follows all objects defined in terms of the metric $\hat{g}$ will be decorated with a hat symbol. There is an obvious scaling map 
$$
\ell : SM \to \widehat{SM}, \quad (x,v) \mapsto \left(x,\frac{v}{\sqrt{\hat{g}(v,v)}}\right)
$$
which is a fibre-bundle isomorphism covering the identity on $M$. As before we define
$$
p(x,v)=\hat{g}(v,v), \qquad r(x,v)=\hat{g}(v,Jv), \quad\text{and}\quad q(x,v)=\hat{g}(Jv,Jv). 
$$ 
\begin{Lemma}\label{lem:pullbackvolume}
The pullback of the volume form $\hat{\Theta}$ on $\widehat{SM}$ is
$$
\ell^*\hat{\Theta}=\left(\frac{pq-r^2}{p}\right)\Theta. 
$$
\end{Lemma}
\begin{proof}
Since 
$$
d\pi\left(X(x,v)\right)=v \quad \text{and}\quad d\pi\left(H(x,v)\right)=Jv,
$$
we obtain
$$
\pi^*\hat{g}=p\omega_1\otimes\omega_1+2r\omega_1\circ\omega_2+q\omega_2\otimes \omega_2,
$$
where we write $\omega_1\circ\omega_2:=\frac{1}{2}\left(\omega_1\otimes\omega_2+\omega_2\otimes\omega_1\right)$. We first compute
\begin{align*}
\left(X\inc \ell^*\hat{\omega}_1\right)(x,v)&=\hat{\omega}_1\left(d\ell(X(x,v))\right)=\hat{g}\left(\ell(x,v),(d\hat{\pi}\circ d\ell)(X(x,v))\right)\\
&=\frac{1}{\sqrt{\hat{g}(v,v)}}\hat{g}(v,\d\pi(X(x,v)))=\sqrt{\hat{g}(v,v)}=\sqrt{p},
\end{align*}
where we have used that $\hat{\pi}\circ\ell=\pi$. Likewise, we obtain
\begin{align*}
\left(H\inc \ell^*\hat{\omega}_1\right)(x,v)&=\hat{\omega}_1\left(d\ell(H(x,v))\right)=\hat{g}\left(\ell(x,v),(d\hat{\pi}\circ d\ell)(H(x,v))\right)\\
&=\frac{1}{\sqrt{\hat{g}(v,v)}}\hat{g}(v,\d\pi(H(x,v)))=\frac{\hat{g}(v,Jv)}{\sqrt{\hat{g}(v,v)}}=\frac{r}{\sqrt{p}}.
\end{align*}
Since $\hat{\omega}_1$ is semibasic for the projection $\hat{\pi}$, the pullback $\ell^*\hat{\omega}_1$ is semibasic for the projection $\pi$, hence $V\inc \ell^*\hat{\omega}_1=0$, so that we have
\begin{equation}\label{eq:pullbackalpha}
\ell^*\hat{\omega}_1=\sqrt{p}\omega_1+\frac{r}{\sqrt{p}}\omega_2. 
\end{equation}
The pullback $\ell^*\hat{\omega}_2$ must be a multiple of $\omega_2$. Indeed, $\ell^*\hat{\omega}_2$ is $\pi$-semibasic and we obtain
\begin{align*}
\left(X\inc \ell^*\hat{\omega}_2\right)(x,v)&=\hat{\omega}_2\left(d\ell(X(x,v))\right)=\hat{g}\left(\hat{J}\ell(x,v),(d\hat{\pi}\circ d\ell)(X(x,v))\right)\\
&=\frac{1}{\sqrt{\hat{g}(v,v)}}\hat{g}(\hat{J}v,\d\pi(X(x,v)))=\frac{\hat{g}(\hat{J}v,v)}{\sqrt{\hat{g}(v,v)}}=0.
\end{align*}
Recall that the area form $\hatarea$ of $\hat{g}$ satisfies $\hat{\pi}^*\hatarea=\hat{\omega}_1\wedge\hat{\omega}_2$, hence
$$
\ell^*(\hat{\omega}_1\wedge\hat{\omega}_2)=\pi^*\hatarea=\sqrt{pq-r^2}\,\omega_1\wedge\omega_2. 
$$
Thus we must have
\begin{equation}\label{eq:pullbackbeta}
\ell^*\hat{\omega}_2=\frac{\sqrt{pq-r^2}}{\sqrt{p}}\omega_2. 
\end{equation}
Since the Lie derivative of $\pi^*\hat{g}$ with respect to $V$ vanishes identically, we compute that $V\sqrt{p}=r/\sqrt{p}$. Moreover, since $\sqrt{pq-r^2}$ is the $\pi$-pullback of a function on $M$, we obtain
$$
V\left(\frac{\sqrt{pq-r^2}}{\sqrt{p}}\right)=-\frac{r\sqrt{pq-r^2}}{p^{3/2}}.
$$
Pulling back the structure equation $d\hat{\omega}_2=-\hat{\psi}\wedge\hat{\omega}_1$ whilst using~\eqref{eq:pullbackalpha} and~\eqref{eq:pullbackbeta} gives
\begin{align*}
\ell^*(d\hat{\omega}_2)&=d(\ell^*\hat{\omega}_2)=d\left(\frac{\sqrt{pq-r^2}}{\sqrt{p}}\omega_2\right)\\
&=\left(\hat{a}\omega_1-\frac{r\sqrt{pq-r^2}}{p^{3/2}}\psi\right)\wedge\omega_2-\frac{\sqrt{pq-r^2}}{\sqrt{p}}\psi\wedge\omega_1\\
&=-\ell^*\hat{\psi}\wedge\ell^*\hat{\omega}_1=-\ell^*\hat{\psi}\wedge\left(\sqrt{p}\omega_1+\frac{r}{\sqrt{p}}\omega_2\right)
\end{align*}
for some unique real-valued function $\hat{a}$ on $SM$. Comparing the coefficients in the above equations, it follows that
\begin{equation}\label{eq:pullbackpsi}
\ell^*\hat{\psi}=a\omega_1+b\omega_2+\frac{\sqrt{pq-r^2}}{p}\psi
\end{equation}
for some unique real-valued functions $a,b$ on $SM$. In particular, we obtain
$$
\ell^*\hat{\Theta}=\ell^*\left(\hat{\omega}_1\wedge\hat{\omega}_2\wedge\hat{\psi}\right)=\left(\frac{pq-r^2}{p}\right)\omega_1\wedge\omega_2\wedge\psi, 
$$
as claimed. 
\end{proof}

We use this lemma to derive the following observation about general thermostats:

\begin{Proposition}\label{prop:transport} If two thermostats determined by pairs $(g,\lambda)$ and $(\hat{g},\hat{\lambda})$ have the same unparametrised geodesics, then
$$
\sqrt{p}\,(\hat{V}\hat{\lambda}\circ\ell)=F\log\left(\frac{pq-r^2}{p^{3/2}}\right)+V\lambda.
$$
\end{Proposition}
As an immediate application we obtain the following classical fact:
\begin{Corollary}
Let $g$ and $\hat{g}$ be two Riemannian metrics on $M$ having the same unparametrised geodesics, then $p/(pq-r^2)^{2/3}$ is an integral for the geodesic flow of $g$. 
\end{Corollary}
\begin{proof}
This special case corresponds to $\lambda=\hat{\lambda}=0$ and hence Proposition~\ref{prop:transport} implies
\begin{align*}
0&=X\log\left(\frac{pq-r^2}{p^{3/2}}\right)=-\frac{3}{2}X\log\left(\frac{p}{(pq-r^2)^{2/3}}\right)\\
&=-\frac{3}{2}\frac{(pq-r^2)^{2/3}}{p}X\left(\frac{p}{(pq-r^2)^{2/3}}\right). 
\end{align*}
\end{proof}

In order to prove Proposition~\ref{prop:transport} we also recall a general lemma whose proof is elementary and thus omitted. 

\begin{Lemma}\label{lem:basicid} Let $X$ be a vector field on a manifold $M$ with volume form $\Omega$. Let $f$ and $s>0$ be smooth functions.
Then
$$
\Div_{\Omega}(fX)=Xf+f\Div_{\Omega}X \quad \text{and}\quad 
\Div_{s\,\Omega}(X)=X\log s+\Div_{\Omega}X.
$$
\end{Lemma}

\begin{proof}[Proof of Proposition~\ref{prop:transport}] This follows from Lemma~\ref{lem:pullbackvolume} and~\ref{lem:basicid} and the key fact that if the thermostats have the same unparametrised
geodesics then
\begin{equation}
\label{eq:keyid}
\ell^*\hat{F}=\frac{1}{\sqrt{p}}\,F.
\end{equation}
To see the last equality, note that we can rephrase the hypothesis as follows. There is a smooth function
$\tau:SM\times\re\to\re$ implementing the time change so that
\[\ell\circ \phi_{\tau(x,v,t)}(x,v)=\hat{\phi}_{t}\circ\ell(x,v).\]
Differentiating this with respect to $t$ and setting $t=0$ gives
\[d\ell(fF)=\hat{F}\circ\ell,\]
where $f(x,v):=\frac{d}{dt}\tau(x,v,t)|_{t=0}$. To check that $f$ has the desired form, apply $d\hat{\pi}$ to the last equation
to get $f\,v=v/\sqrt{\hat{g}(v,v)}$. 

Writing $s:=(pq-r^2)/p$ and taking the divergence of~\eqref{eq:keyid} with respect to $\ell^*\hat{\Theta}=s\Theta$ gives
\begin{align*}
\Div_{s\,\Theta}\left(\sqrt{p}\,\ell^*\hat{F}\right)&=(\ell^*\hat{F})\sqrt{p}+\sqrt{p}\,\Div_{s\,\Theta}(\ell^*\hat{F})\\
&=\left(\frac{1}{\sqrt{p}}\right)F\sqrt{p}+\sqrt{p}\,\Div_{\ell^*\hat{\Theta}}\left(\ell^*\hat{F}\right)\\
&=F\left(\log\sqrt{p}\right)+\sqrt{p}\left(\Div_{\hat{\Theta}}\hat{F}\right)\circ \ell\\
&=\Div_{s\,\Theta}F=F\log s+\Div_{\Theta}F
\end{align*}
where we have used Lemma~\ref{lem:basicid}. Since $\Div_{\Theta}F=V\lambda$ and $\Div_{\hat{\Theta}}\hat{F}=\hat{V}\hat{\lambda}$ this last equation is equivalent to
$$
\sqrt{p}\left(\hat{V}\hat{\lambda}\circ \ell\right)=F\log\left(\frac{s}{\sqrt{p}}\right)+V\lambda,  
$$
which proves the claim.
\end{proof}
\begin{Remark}
Note that the crucial identity~\eqref{eq:keyid} also follows from a different argument. Since the orbits of $F$ and $\hat{F}$ project onto the same unparametrised curves, there must exist a smooth function $w$ on $SM$, so that $\ell^*\hat{F}=wF$. From~\eqref{eq:pullbackalpha},~\eqref{eq:pullbackbeta} and~\eqref{eq:pullbackpsi}, we compute that
$$
\ell^*\hat{X}=\frac{1}{\sqrt{p}}X-\frac{a\sqrt{p}}{\sqrt{pq-r^2}}V\quad \text{and}\quad \ell^*\hat{V}=\frac{p}{\sqrt{pq-r^2}}V
$$
from which one immediately obtains $w=1/\sqrt{p}$. 
\end{Remark}
A special case of Proposition~\ref{prop:transport} is the following:
\begin{Corollary}\label{cor:transportweyl}
Suppose the projective thermostat associated to the pair $(g,\lambda)=(g,a-V\theta)$ has the same unparametrised geodesics as the Weyl connection $\D$ defined by $(\hat{g},\alpha)$, then 
$$
u=\frac{3}{2}\log\left(\frac{p}{(pq-r^2)^{2/3}}\right)
$$
satisfies the transport equation
\begin{equation}\label{eq:transportpde}
Fu=Va+\beta,
\end{equation}
where $\beta=\theta-\alpha$. 
\end{Corollary}
\begin{proof}
Applying Proposition~\ref{prop:transport} in the special case $\lambda=a-V\theta$ and $\hat{\lambda}=-\hat{V}\alpha$ gives
$$
-\sqrt{p}\left(\hat{V}\hat{V}\alpha\circ\ell\right)=\sqrt{p}\left(\alpha\circ\ell\right)=F\log\left(\frac{pq-r^2}{p^{3/2}}\right)+V(a-V\theta),
$$
the left hand side of which is simply $\alpha$, thought of as a function on $SM$. Hence we obtain
\begin{align*}
-\left(Va+\theta-\alpha\right)&=F\log\left(\frac{pq-r^2}{p^{3/2}}\right)=F\left(-\frac{3}{2}\left(\log p-\frac{2}{3}\log(pq-r^2)\right)\right)\\
&=-Fu, 
\end{align*}
as claimed. 
\end{proof}

\section{The tensor tomography result}

In this final section we prove a vanishing theorem for the transport equation $Fu=Va+\beta$, provided the triple $(g,A,\theta)$ defining $F$ satisfies certain conditions. Recall that every properly convex projective structure $\mathfrak{p}$ arises from a triple $(g,A,0)$ satisfying 
$$
K_g=-1+2|A|^2_g \quad \text{and} \quad \ov{\partial} A=0. 
$$
In particular, we would like to conclude that if such a $\mathfrak{p}$ contains a Weyl connection, then $A$ must vanish identically and hence $\mathfrak{p}$ is hyperbolic. It turns out that one can prove a more general vanishing theorem for a class of thermostats arising from a triple $(g,A,\theta)$ where $A$ is a differential of degree $m\geqslant 3$ on $M$, that is, a section of $\can^m$. Suppose $A \in \Gamma(\can^m)$. Like in the case $m=3$ there exists a unique smooth real-valued function $a$ on $SM$ lying in $\mathcal{H}_{-m}\oplus \mathcal{H}_m$, so that $\pi^*A=(V(a)/m+\i a)\omega^m$. In particular, to a triple $(g,A,\theta)$ we may associate the thermostat $F=X+(a-V\theta)V$. We now have:

\begin{Theorem}Let $M$ be a closed oriented surface and $(g,A,\theta)$ be a triple satisfying
$$\ov{\partial}A=\left(\frac{m-1}{2}\right)\left(\theta-i\star_g\theta\right)\otimes A\quad \text{and} \quad
K_g-\delta_{g}\theta+(2-m)|A|^{2}_{g}\leqslant 0.$$ Let $F$ denote the vector field of the thermostat determined by $(g,A,\theta)$.
Suppose there is a 1-form $\beta\in \Omega^1(M)$ and a function $u\in C^{\infty}(SM)$ such that
\[Fu=Va+\beta.\]
Then $A=0$ and $\beta$ is exact.
\label{thm:noweylV3}
\end{Theorem}
Let us first verify that this gives the desired statement. 
\begin{Corollary}\label{cor:nonweylpropconvex}
Let $(M,\mathfrak{p})$ be a closed oriented properly convex projective surface with $\chi(M)<0$ and with $\mathfrak{p}$ containing a Weyl connection $\weyl$. Then $\mathfrak{p}$ is hyperbolic and moreover $\weyl$ is the Levi-Civita connection of the hyperbolic metric.  
\end{Corollary}
\begin{proof}
By a result of Calabi~\cite{MR0365607}, if $m=3$ and $(g,A)$ satisfy
$$
K_g=-1+2|A|^2_g \quad \text{and} \quad \ov{\partial} A=0,
$$
then $K_g\leqslant 0$. In particular, the triple $(g,A,0)$ satisfies the assumptions of Theorem~\ref{thm:noweylV3} and Corollary~\ref{cor:transportweyl} implies that we have a solution $u$ to the transport equation $Fu=Va+\beta$. Hence the theorem gives right away that $A$ vanishes identically and hence $\mathfrak{p}$ is hyperbolic. In particular, the Levi-Civita connection ${}^g\nabla$ of the hyperbolic metric and the connection $\weyl$ both lie in $\mathfrak{p}$ and hence are projectively equivalent, but this can happen if and only if ${}^{g}\nabla=\weyl$, by Corollary~\ref{cor:georig}.   
\end{proof}

\begin{Remark}
In~\cite{arXiv:1609.08033} the notion of a minimal Lagrangian connection is introduced. These are torsion-free connections on $TM$ of the form $\nabla=\weyl+B$ where $(g,A,\theta)$ defining $\weyl$ and $B$ are subject to the equations
$$
K_g-\delta_g\theta=-1+2|A|^2_g, \qquad \ov{\partial} A=\left(\theta-\i\star_g\theta\right)\otimes A, \qquad \d\theta=0. 
$$
In particular, on a closed oriented surface of negative Euler characteristic every properly convex projective structure arises from a minimal Lagrangian connection. Another immediate consequence of Theorem~\ref{thm:noweylV3} and Corollary~\ref{cor:georig} thus is:
\end{Remark}
\begin{Corollary}\label{cor:nonweylminlag}
Let $M$ be a closed oriented surface of negative Euler characteristic and $\nabla$ a minimal Lagrangian connection arising from the triple $(g,A,\theta)$. Suppose $|A|^2_g\leqslant 1$ and that $\nabla$ is projectively equivalent to a Weyl connection $\weyl$. Then $A$ vanishes identically and hence $\nabla=\weyl$.  
\end{Corollary}

In order to show the theorem we use the following $L^2$ identity proved in \cite[Equation (5)]{MR2486586} which is in turn an extension of an identity in \cite{MR1863022} for geodesic flows. The identity holds for arbitrary thermostats $F=X+\lambda V$. If we let $H_{c}:=H+cV$ where $c:SM\to\re$ is any smooth function then
\begin{equation}
2\langle H_{c}u,VFu\rangle=\|Fu\|^{2}+\|H_{c}u\|^{2}-\langle Fc+c^{2}+K_g-H_{c}\lambda+\lambda^{2},(Vu)^{2}\rangle,
\label{eq:l2}
\end{equation}
where $u$ is any smooth function. All norms and inner products are $L^{2}$ with respect to the volume form $\Theta$.

We also need the following lemma whose proof is a straightforward calculation (see~\cite[Lemma 4.1]{arXiv:1706.03554} for a proof).

\begin{Lemma}
We have
$$
\ov{\partial}A=\left(\frac{m-1}{2}\right)(\theta-\i\star_g\theta)\otimes A
$$
if and only if
$$
XVa-mHa-(m-1)(\theta Va-maV\theta)=0. 
$$
\end{Lemma}

\begin{proof}[Proof of Theorem~\ref{thm:noweylV3}] Without loss of generality we may assume that $\beta$ has zero divergence. Indeed if not, a standard application of scalar elliptic PDE theory shows that we can always
find a smooth function $h$ on $M$ such that $\beta+dh$ has zero divergence. Now note that
$F(u+h)=Va+\beta+dh$.

A calculation shows that if we pick $c=\theta+V(a)/m$, then
$$
Fc+c^{2}+K_g-H_{c}\lambda+\lambda^{2}=K_g-\delta_{g}\theta+(1-m)|A|_{g}^{2},
$$
where we use that
$$
\pi^*|A|^2_g=(Va)^2/m^2+a^2\quad \text{and} \quad \pi^*\delta_g\theta=-\left(X\theta+HV\theta\right),
$$
hence for this choice of $c$, \eqref{eq:l2} simplifies to
\begin{multline}
2\langle H_{c}u,VFu\rangle-\||A|_{g}Vu\|^{2}\\
=\|Fu\|^{2}+\|H_{c}u\|^{2}-\langle K_{g}-\delta_{g}\theta+(2-m)|A|_{g}^{2},(Vu)^2\rangle.
\label{eq:l44}
\end{multline}
If $Fu=Va+\beta$, then $VFu=-m^{2}a+V\beta$. Using that $X$ and $H$ preserve $\Theta$ and that
$XVa-mHa-(m-1)(\theta Va-maV\theta)=0$ we compute
\begin{align*}
2\left\langle H_{c}u,-m^{2}a\right\rangle&=-2m^{2}\langle Hu,a\rangle-2m^{2}\langle cVu,a\rangle\\
&=2m^{2}\langle u,Ha\rangle -2m^{2}\langle cVu,a\rangle\\
&=-2m^{2}\langle Xu,V(a)/m\rangle-2m(m-1)\langle u, \theta Va-maV\theta\rangle\\
&\phantom{=}\;-2m^{2}\langle cVu,a\rangle\\
&=-2m\|Va\|^{2}=-2m^3\|a\|^2,
\end{align*}
where the last equation is obtained using that $Xu=\beta+Va-(a-V\theta)Vu$, $\langle \beta,Va\rangle=0$ and $c=\theta+V(a)/m$.

Using that $X$ and $H$ preserve $\Theta$ and that
$X\beta+HV\beta=0$ ($\beta$ is assumed to have zero divergence) we compute:
\begin{align*}
2\left\langle H_{c}u,V\beta\right\rangle&=2\langle Hu,V\beta\rangle+2\langle cVu,V\beta\rangle\\
&=-2\langle u,HV\beta\rangle +2\langle cVu,V\beta\rangle\\
&=-2\langle Xu,\beta\rangle+2\langle cVu,V\beta\rangle\\
&=-2\|\beta\|^{2}+2\langle (a-V\theta)Vu,\beta\rangle+2\langle cVu,V\beta\rangle\\
&=-2\|\beta\|^2+2\langle aVu,\beta\rangle+2\langle (Va Vu)/m,V\beta\rangle,
\end{align*}
where the penultimate equation is obtained using that $Xu=\beta+Va-(a-V\theta)Vu$ and $\langle \beta,Va\rangle=0$. The last equation uses that $c=\theta+V(a)/m$ and
\[V(\theta V\beta-V\theta\beta)=0.\]

Inserting these calculations back into \eqref{eq:l44}, we derive
\begin{multline*}
-2m^3\|a\|^2-2\|\beta\|^2+2\langle aVu,\beta\rangle+2\langle (Va Vu)/m,V\beta\rangle-\||A|_{g}Vu\|^{2}\\
=\|Fu\|^{2}+\|H_{c}u\|^{2}-\langle K_{g}-\delta_{g}\theta+(2-m)|A|_{g}^{2},(Vu)^2\rangle.
\end{multline*}
Since $|A|_{g}^2=a^2+(Va)^{2}/m^2$ this can be re-written as
\begin{multline*}
-2m^3\|a\|^2-\|\beta-aVu\|^2-\|V\beta-VaVu/m\|^{2}\\
=\|Fu\|^{2}+\|H_{c}u\|^{2}-\langle K_{g}-\delta_{g}\theta+(2-m)|A|_{g}^{2},(Vu)^2\rangle,
\end{multline*}
where we have used that $\|\beta\|^2=\|V\beta\|^2$. By hypothesis the right hand side is $\geqslant 0$ which gives right away that $a=\beta=0$.
\end{proof}

\appendix

\section{Deriving the transport equation}\label{app:1}

Here we sketch how to derive the transport equation for the function $u$ starting from the PDE
$$
\D^{\prime\prime}\mu-\mu\,\D^{\prime}\mu=\Phi\mu^3+\ov{\Phi}
$$
for the Beltrami differential $\mu$. Let $(g,A,\theta)$ be the triple encoding $\mathfrak{p}$ so that the connection form of $\D$ on $SM$ is (see~\eqref{eq:defconfconmatrix}) $\kappa=\i\psi-2\theta_1\omega$, where we write $\theta_1=\frac{1}{2}(\theta-\i V\theta)$. Moreover, on $SM$ the section $\Phi$ of $K^2\otimes \ov{\acan}$ is represented by $a_3=\frac{1}{3}Va+\i a$, where $a(v)=\Re A(Jv,Jv,Jv)$, $v \in SM$. Writing $\mu_{-2}$ for the complex-valued function on $SM$ representing the Beltrami differential $\mu$ and $\mu_2=\ov{\mu_{-2}}$, the PDE for $\mu$ is equivalent to
$$
\d \mu_{-2}=\mu_{-2}^{\prime}\omega+\left(\mu_{-2}\mu_{-2}^{\prime}+a_3\mu_{-2}^3+\ov{a_3}\right)\ov{\omega}+\ov{\kappa}\mu_{-2}-\kappa\mu_{-2},
$$
where $\mu_{-2}^{\prime}$ is a complex-valued function on $SM$. Since $\mu_{-2}$ represents a section of $\ov{K}\otimes \acan\simeq K^{-2}$, writing $\eta_{\pm}=\frac{1}{2}\left(X\mp \i H\right)$ we also have
$$
\d \mu_{-2}=\eta_+(\mu_{-2})\omega+\eta_{-}(\mu_{-2})\ov{\omega}-2 \i \mu_{-2} \psi.
$$
Thus the PDE is equivalent to the system
\begin{equation}\label{eq:appendixone}
\eta_{-}\mu_{-2}-\mu_{-2}\eta_+\mu_{-2}=a_3\mu_{-2}^3-2\mu_{-2}^2\theta_1-2\mu_{-2}\ov{\theta_{1}}+\ov{a_3}
\end{equation}
and $V\mu_{-2}=-2\i\mu_{-2}$. The Beltrami differential does only define a conformal equivalence class $[\hat{g}]$. We may fix a metric $\hat{g}\in [\hat{g}]$ by requiring
$$
\frac{1}{2}\left(p+q\right)=\frac{1+|\mu_2|^2}{\left(1-|\mu_2|^2\right)^4},
$$
where again we specify the metric $\hat{g}$ in terms of the functions $p,q,r$. Explicitly, we have
$$
\frac{1}{2}(p-q)=\frac{\mu_{-2}+\mu_2}{\left(1-|\mu_2|^2\right)^4} \quad \text{and} \quad r=\frac{\i(\mu_2-\mu_{-2})}{\left(1-|\mu_2|^2\right)^4}.
$$
In particular, this yields
$$
h:=\frac{p}{(pq-r^2)^{2/3}}=(\mu_{-2}+1)(\mu_2+1).
$$
Writing $F=X+(a-V\theta)V$ and using~\eqref{eq:appendixone}, a lenghty but straightforward calculation shows that
$$
Fh=\frac{2}{3}hVa+2h\Re\left(a_3\mu_{-2}^2-\mu_2a_{-3}-2\mu_2\theta_{-1}+\eta_+\mu_{-2}\right).
$$
Hence if we define $u=\frac{3}{2}\log h$, then we obtain
$$
Fu-Va=3\Re\left(a_3\mu_{-2}^2-\mu_2a_{-3}-2\mu_2\theta_{-1}+\eta_+\mu_{-2}\right)
$$
Note that the right hand side of the last equation lies in $\mathcal{H}_{-1}\oplus \mathcal{H}_1$, hence there exists a $1$-form $\beta$ on $M$ so that
$$
Fu=Va+\beta
$$
which is the transport equation~\ref{eq:transportpde}. 


\begin{thebibliography}{10}

\bibitem{MR867684}
\bgroup\scshape{}A.~L. Besse\egroup{}, \emph{Einstein manifolds},
  \emph{Ergebnisse der Mathematik und ihrer Grenzgebiete (3)} \textbf{10},
  Springer-Verlag, Berlin, 1987. \mr{867684}\;  \zbl{0613.53001}\;

\bibitem{MR2581355}
\bgroup\scshape{}R.~L. Bryant\egroup{}, \bgroup\scshape{}M.~Dunajski\egroup{},
  and \bgroup\scshape{}M.~Eastwood\egroup{}, Metrisability of two-dimensional
  projective structures,  \emph{J. Differential Geom.} \textbf{83} (2009),
  465--499. \mr{2581355}\;  \zbl{1196.53014}\;

\bibitem{MR0365607}
\bgroup\scshape{}E.~Calabi\egroup{}, Complete affine hyperspheres. {I},
  {S}ymposia {M}athematica, {V}ol. {X} ({C}onvegno di {G}eometria
  {D}ifferenziale, {INDAM}, {R}ome, 1971),  (1972), 19--38. \mr{0365607}\;
  \zbl{0252.53008}\;

\bibitem{MR1632920}
\bgroup\scshape{}C.~B. Croke\egroup{} and \bgroup\scshape{}V.~A.
  Sharafutdinov\egroup{}, Spectral rigidity of a compact negatively curved
  manifold,  \emph{Topology} \textbf{37} (1998), 1265--1273. \mr{1632920}\;
  \zbl{0936.58013}\;

\bibitem{MR728412}
\bgroup\scshape{}M.~Dubois-Violette\egroup{}, Structures complexes au-dessus
  des vari\'et\'es, applications,  in \emph{Mathematics and physics},
  \emph{Progr. Math.} \textbf{37}, Birkh\"auser Boston, Boston, MA, 1983,
  pp.~1--42. \mr{728412}\;  \zbl{0522.53029}\;

\bibitem{MR3432157}
\bgroup\scshape{}D.~Dumas\egroup{} and \bgroup\scshape{}M.~Wolf\egroup{},
  Polynomial cubic differentials and convex polygons in the projective plane,
  \emph{Geom. Funct. Anal.} \textbf{25} (2015), 1734--1798. \mr{3432157}\;
  \zbl{06526259}\;

\bibitem{MR579579}
\bgroup\scshape{}V.~Guillemin\egroup{} and
  \bgroup\scshape{}D.~Kazhdan\egroup{}, Some inverse spectral results for
  negatively curved {$2$}-manifolds,  \emph{Topology} \textbf{19} (1980),
  301--312. \mr{579579}\;  \zbl{0465.58027}\;

\bibitem{MR699802}
\bgroup\scshape{}N.~Hitchin\egroup{}, Complex manifolds and {E}instein's
  equations,  in \emph{Twistor geometry and nonlinear systems}, \emph{Lecture
  Notes in Math.} \textbf{970}, Springer, Berlin, 1982, pp.~73--99.
  \mr{699802}\;  \zbl{0507.53025}\;

\bibitem{MR1174252}
\bgroup\scshape{}N.~Hitchin\egroup{}, Lie groups and {T}eichm\"uller space,
  \emph{Topology} \textbf{31} (1992), 449--473. \mr{1174252}\;
  \zbl{0769.32008}\;

\bibitem{MR2486586}
\bgroup\scshape{}D.~Jane\egroup{} and \bgroup\scshape{}G.~P.
  Paternain\egroup{}, On the injectivity of the {X}-ray transform for {A}nosov
  thermostats,  \emph{Discrete Contin. Dyn. Syst.} \textbf{24} (2009),
  471--487. \mr{2486586}\;  \zbl{1161.37315}\;

\bibitem{MR2402597}
\bgroup\scshape{}F.~Labourie\egroup{}, Flat projective structures on surfaces
  and cubic holomorphic differentials,  \emph{Pure Appl. Math. Q.} \textbf{3}
  (2007), 1057--1099. \mr{2402597}\;  \zbl{1158.32006}\;

\bibitem{MR1828223}
\bgroup\scshape{}J.~C. Loftin\egroup{}, Affine spheres and convex
  {$\mathbb{RP}^n$}-manifolds,  \emph{Amer. J. Math.} \textbf{123} (2001),
  255--274. \mr{1828223}\;  \zbl{0997.53010}\;

\bibitem{MR1796527}
\bgroup\scshape{}V.~S. Matveev\egroup{} and \bgroup\scshape{}P.~J.
  Topalov\egroup{}, Metric with ergodic geodesic flow is completely determined
  by unparameterized geodesics,  \emph{Electron. Res. Announc. Amer. Math.
  Soc.} \textbf{6} (2000), 98--104. \mr{1796527}\;  \zbl{0979.53032}\;

\bibitem{MR3144212}
\bgroup\scshape{}T.~Mettler\egroup{}, Weyl metrisability of two-dimensional
  projective structures,  \emph{Math. Proc. Cambridge Philos. Soc.}
  \textbf{156} (2014), 99--113. \mr{3144212}\;  \zbl{06283094}\;

\bibitem{arXiv:1510.01043}
\bgroup\scshape{}T.~Mettler\egroup{}, Extremal conformal structures on
  projective surfaces, 2015, to appear in~\textit{Ann. Scuola Norm. Sup. Pisa
  Cl. Sci. (5)}. \arxiv{1510.01043}

\bibitem{MR3384876}
\bgroup\scshape{}T.~Mettler\egroup{}, Geodesic rigidity of conformal
  connections on surfaces,  \emph{Math. Z.} \textbf{281} (2015), 379--393.
  \mr{3384876}\;  \zbl{1326.53018}\;

\bibitem{arXiv:1609.08033}
\bgroup\scshape{}T.~Mettler\egroup{}, Minimal {L}agrangian connections on
  compact surfaces, 2016. \arxiv{1609.08033}

\bibitem{arXiv:1706.03554}
\bgroup\scshape{}T.~Mettler\egroup{} and \bgroup\scshape{}G.~P.
  Paternain\egroup{}, Holomorphic differentials, thermostats and {A}nosov
  flows, \emph{Math.~Ann.} \textbf{373} (2019), 553--580.

\bibitem{MR812312}
\bgroup\scshape{}N.~R. O'Brian\egroup{} and \bgroup\scshape{}J.~H.
  Rawnsley\egroup{}, Twistor spaces,  \emph{Ann. Global Anal. Geom.} \textbf{3}
  (1985), 29--58. \mr{812312}\;  \zbl{0526.53057}\;

\bibitem{MR3069117}
\bgroup\scshape{}G.~P. Paternain\egroup{}, \bgroup\scshape{}M.~Salo\egroup{},
  and \bgroup\scshape{}G.~Uhlmann\egroup{}, Tensor tomography on surfaces,
  \emph{Invent. Math.} \textbf{193} (2013), 229--247. \mr{3069117}\;
  \zbl{06197114}\;

\bibitem{MR3263517}
\bgroup\scshape{}G.~P. Paternain\egroup{}, \bgroup\scshape{}M.~Salo\egroup{},
  and \bgroup\scshape{}G.~Uhlmann\egroup{}, Spectral rigidity and invariant
  distributions on {A}nosov surfaces,  \emph{J. Differential Geom.} \textbf{98}
  (2014), 147--181. \mr{3263517}\;  \zbl{1304.37021}\;

\bibitem{MR1863022}
\bgroup\scshape{}V.~Sharafutdinov\egroup{} and
  \bgroup\scshape{}G.~Uhlmann\egroup{}, On deformation boundary rigidity and
  spectral rigidity of {R}iemannian surfaces with no focal points,  \emph{J.
  Differential Geom.} \textbf{56} (2000), 93--110. \mr{1863022}\;
  \zbl{1065.53039}\;

\bibitem{spivakvol2}
\bgroup\scshape{}M.~Spivak\egroup{}, \emph{A comprehensive introduction to
  differential geometry. {V}ol. {II}}, third ed., Publish or Perish, Inc.,
  Wilmington, Del., 1999.

\bibitem{MR1178538}
\bgroup\scshape{}C.~P. Wang\egroup{}, Some examples of complete hyperbolic
  affine {$2$}-spheres in {${\bf R}^3$},  in \emph{Global differential geometry
  and global analysis ({B}erlin, 1990)}, \emph{Lecture Notes in Math.}
  \textbf{1481}, Springer, Berlin, 1991, pp.~271--280. \mr{1178538}\;
  \zbl{0743.53004}\;

\end{thebibliography}

\providecommand{\noopsort}[1]{}
\providecommand{\mr}[1]{\href{http://www.ams.org/mathscinet-getitem?mr=#1}{MR~#1}}
\providecommand{\zbl}[1]{\href{http://www.zentralblatt-math.org/zmath/en/search/?q=an:#1}{Zbl~#1}}
\providecommand{\arxiv}[1]{\href{http://www.arxiv.org/abs/#1}{arXiv:#1}}
\providecommand{\doi}[1]{\href{http://dx.doi.org/#1}{DOI}}
\providecommand{\MR}{\relax\ifhmode\unskip\space\fi MR }
\providecommand{\MRhref}[2]{%
  \href{http://www.ams.org/mathscinet-getitem?mr=#1}{#2}
}
\providecommand{\href}[2]{#2}

\end{document}